\documentclass[12pt]{amsart}
\usepackage[dvipsnames]{color}
\usepackage{amsfonts,amssymb,amsmath,amscd,amstext}
\usepackage[colorlinks=true,linkcolor=teal,citecolor=purple]{hyperref}
\usepackage[utf8]{inputenc}
\usepackage{microtype}
\usepackage{graphicx}
\usepackage{changes}
\usepackage[noabbrev,capitalize,nameinlink]{cleveref}
\usepackage{comment}
\usepackage{mathdots}
\usepackage[a4paper,top=2.5cm,bottom=2.5cm,left=2cm,right=2cm]{geometry}
\renewcommand{\leq}{\leqslant}
\renewcommand{\geq}{\geqslant}

\renewcommand{\v}{\nu^\eps}
\newcommand{\g}[1]{g_\varepsilon\left( #1 \right)}
\newcommand{\ptl}{\partial}
\newcommand{\hhh}{{\mathcal{H}}}

\newcommand{\ve}{\nu^\varepsilon}
\newcommand{\n}[1]{\nabla_{#1}}
\newcommand{\hn}{\mathbb{H}^n}
\newcommand{\vh}{\nu^\hh}
\newcommand{\vet}{\ve_{2n+1}}
\newcommand{\vt}{\ve_{2n+1}}

\newcommand{\rr}{{\mathbb{R}}}

\newcommand{\hh}{{\mathbb{H}}}

\newcommand{\Om}{\Omega}
\newcommand{\eps}{\varepsilon}

\newcommand{\ga}{\gamma}

\newcommand{\escpr}[1]{\langle#1\rangle}

\newcommand{\scu}{\longrightarrow}

\newcommand{\z}{W^{\varphi,\varepsilon}}

\definechangesauthor[color=red]{J}
\definecolor{champagne}{rgb}{0.97, 0.91, 0.81}

\definecolor{asparagus}{rgb}{0.53, 0.66, 0.42}
\DeclareMathOperator{\ric}{Ric}
\DeclareMathOperator{\graf}{graph}
\DeclareMathOperator{\divv}{div}

\DeclareMathOperator{\spann}{span}
\DeclareMathOperator{\lip}{Lip}
\DeclareMathOperator{\supp}{supp}
\newtheorem{theorem}{Theorem}[section]
\newtheorem{proposition}[theorem]{Proposition}
\newtheorem{definition}[theorem]{Definition}
\newtheorem{lemma}[theorem]{Lemma}

\theoremstyle{definition}

\theoremstyle{remark}

\numberwithin{equation}{section}

\definechangesauthor[color=blue]{JP}
\definechangesauthor[color=purple]{S}

\author[J.~Pozuelo]{Juli\'an Pozuelo} 
\address[Juli\'an Pozuelo]{Dipartamento di Matematica “Tullio Levi-Civita” \\
Università di Padova \\ Via Trieste, 63, 35131, Padova, Italy}
\email{julian.pozuelodominguez@unipd.it}

\author[S. Verzellesi]{Simone Verzellesi}
\address[Simone Verzellesi]{Department of Mathematics, University of Trento, Via Sommarive 14, 38123 Povo (Trento), Italy}
\email[Simone Verzellesi]{simone.verzellesi@unitn.it}
\title[$t$-graphs of prescribed mean curvature in Heisenberg groups]{Existence and uniqueness of $t$-graphs of prescribed mean curvature in Heisenberg groups}

\date{\today}

\subjclass{53C17, 49Q10, 53A10}
\keywords{Prescribed mean curvature, Heisenberg group, sub-Riemannian geometry}
\thanks{\textit{Acknowledgements}. The authors would like to thank L. Capogna, G. Giovannardi, A. Pinamonti, M. Ritoré, C. Rosales and G. Saracco for stimulating and fruitful conversations about the contents of this paper.
J. Pozuelo is supported by the research grant PID2020-118180GB-I00 funded by MCIN/AEI/10.13039/501100011033 and  Fundación Ramón Areces grants \emph{XXXV convocatoria para la amplación de estudios en el extranjero en ciencias de la vida y la materia}.
S. Verzellesi is member of the Istituto Nazionale di Alta Matematica (INdAM), Gruppo Nazionale per l'Analisi Matematica, la Probabilità e le loro Applicazioni (GNAMPA), and is supported by INdAM–GNAMPA 2023 Project \emph{Equazioni differenziali alle derivate parziali di tipo misto o dipendenti da campi di vettori} and by  MIUR-PRIN 2022 Project \emph{Regularity problems in sub-Riemannian structures}. }
\begin{document} 

\maketitle

 \dedicatory\centerline{In Memory of Enrico Giusti}
 
\begin{abstract}
 We study the prescribed mean curvature equation for $t$-graphs in a Riemannian Heisenberg group of arbitrary dimension. We characterize the existence of classical solutions in a bounded domain without imposing Dirichlet boundary data, and we provide conditions that guarantee uniqueness. Moreover, we extend previous results to solve the Dirichlet problem when the mean curvature is non-constant. Finally, by an approximation technique, we obtain solutions to the sub-Riemannian prescribed mean curvature equation.
\end{abstract}


\section{Introduction}

The Plateau problem has been a fundamental issue in geometry since the pioneering works of Douglas (cf. \cite{MR1501590}) and Radó (cf. \cite{MR1502955}). The Euclidean prescribed mean curvature equation of the graph of a function $u\in C^2(\Om)$ over a bounded domain $\Om\subseteq\rr^n$ with $H\in C(\Om)$ reads as
\begin{equation}\tag{PMC}\label{PMC}
\divv\left(\frac{Du}{\sqrt{1+|Du|^2}}\right)=H,
\end{equation}
where $D$ stands for the Euclidean gradient. Notice that we are taking the mean curvature as the (not averaged) sum of the principal curvatures. When $H$ is constant and $\partial\Om$ is of class $C^2$, Serrin (cf. \cite{MR282058}) characterized the existence of solutions to the Dirichlet problem 
for \emph{any} boundary datum $\varphi\in C^2(\ptl\Om)$ by the condition
\begin{equation}\label{Serrin}
|H|\leq H_{\ptl\Om}(z_0)
\end{equation}
for any $z_0\in \ptl\Om$, where $H_{\ptl\Om}$ is the mean curvature of the boundary of $\Om$. In the proof, Serrin obtained Schauder estimates for $C^2$ solutions first by providing height estimates for $|u|$, and then, by means of a gradient maximum principle, showing that the maximum of the gradient is attained at the boundary of $\Om$. In the final step, he estimated the gradient at the boundary exploiting the so-called \emph{barriers} (cf. \cite{GT}), whose construction relies on \eqref{Serrin}. 
When $H$ is not constant, an approach based on the maximum principle typically fails. Therefore, in order to deal both with non-constant sources and to allow merely continuous boundary data, it is customary to rely on suitable interior and global gradient estimates.
Some references to these kind of estimates in the Euclidean space are the works of Korevaar and Simon (cf. \cite{MR0932680}) and Wang (cf. \cite{MR1617971}). Beyond the Euclidean framework, the Dirichlet problem for any sufficiently regular boundary datum and constant source $H$ satisfying conditions analogous to \eqref{Serrin} has been studied in warped products with a particular lower bound on the Ricci curvature (cf. \cite{MR2351645}), in $3$-dimensional Heisenberg groups (cf. \cite{MR2332426}) and in higher dimensional Heisenberg groups (cf. \cite{pmc1}). When $H$ is not constant, the previous results were later extended in Riemannian manifolds with a Killing vector field and a lower bound on the Ricci curvature depending on $\Om$ (cf. \cite{MR3540594,MR2526401,MR2413108}). 
When instead \eqref{Serrin} fails, meaning that
\begin{equation}\label{Serrinnot}
|H|> H_{\ptl\Om}(z_0)
\end{equation}
for some $z_0\in\partial\Om$,
we could lose control of the norm of the gradient of a solution near the boundary, and hence of the existence of solutions. More precisely, as shown in \cite{GT}, when \eqref{Serrinnot} holds there always exists a boundary datum $\varphi$ for which the Dirichlet problem has no solution. 
Nevertheless, the validity of \eqref{Serrinnot} does not preclude \emph{a priori} the existence of a \emph{suitable} boundary datum $\varphi$ for which the Dirichlet problem is solvable. As an instance, taking as domain $\Om\subseteq\rr^n$ the ball of radius $1$ centered at $0$, we can write the half sphere in $\rr^{n+1}$ centered at $0$ with radius $1$ as a graph over $\Om$. A simple computation reveals that it satisfies \eqref{PMC} with $H=\frac{n}{n-1}H_{\partial\Om}$ and with boundary datum $\varphi\equiv 0$. In particular, \eqref{Serrinnot} is verified for any $z_0\in\partial \Om$.  
In this regard, when $\Om$ has Lipschitz boundary, Giusti (cf. \cite{MR487722}) proved that the existence of solutions to \eqref{PMC} with a suitable boundary condition, not imposed \emph{a priori}, is characterized by
\begin{equation}\tag{*}\label{Giusti}
\left|\int_{\tilde \Om} H(x)\,dx\right|<P(\tilde \Om)
\end{equation}
for any set $\tilde \Om\subseteq\Om$ such that $\tilde \Om\neq\emptyset$ and $\tilde \Om\neq\Om$, where $P(\tilde \Om)$ is the perimeter of $\tilde \Om$. Moreover, \cite{MR487722} provides a characterization of those domains where \eqref{PMC} admits, up to vertical translations, a unique solution. Precisely, the previous statement is equivalent to each of the following conditions: there is no solution to \eqref{PMC} in any domain $\Om\subsetneq\hat \Om$; there is a solution on $\Om$ which is vertical at every point of $\ptl\Om$; \eqref{Giusti} holds and
\begin{equation}\tag{**}\label{Giusti2}
\left| \int_\Om H(x)\,dx \right|=P(\Om).
\end{equation}
In these cases, $\Om$ is called an \emph{extremal domain}. Otherwise, i.e. when \eqref{Giusti} also holds for $\tilde\Om=\Om$, then $\Om$ is called a \emph{non-extremal domain}. The proof of the existence of solutions under condition \eqref{Giusti} relies on previous results by Giaquinta (cf. \cite{MR336532,MR0377669}) and Miranda (cf. \cite{MR174706}). In the non-extremal case the proof consists in showing the existence of $BV$ minimizers of the penalized functional
\begin{equation}\label{penafun}
    \mathcal{F}(u)=\int_\Om \sqrt{1+|D u|^2}+\int_\Om Hu\,dx+\int_{\partial\Om} |u-\varphi|\,d\mathcal{H}^{n-1}
\end{equation}
for any $\varphi\in L^1(\ptl\Om)$, whose regularity is then gradually improved in several steps. 
The more involved extremal case, i.e. when \eqref{Giusti2} holds, follows by a compactness procedure. More precisely, in view of condition \eqref{Giusti}, every domain $\tilde\Om\subsetneq\Om$ is itself a non-extremal domain. Therefore, exploiting the existence result in the non-extremal case, together with a compactness argument based on a notion of \emph{generalized solution} first introduced by Miranda (cf. \cite{MR0500423}), existence in the extremal case follows. {For similar results under weaker assumptions on the boundary of $\Om$ we refer the reader to \cite{MR3767675}.\\

Although the Dirichlet problem, as previously discussed, has been widely studied beyond the Euclidean framework,
as far as the authors are aware no results in the spirit of \cite{MR487722} are available in the Riemannian setting. The aim of this work, consequently, is to lay the groundwork for the study of hypersurfaces with prescribed mean curvature outside the Euclidean setting and overcoming conditions inspired by \eqref{Serrin}.
As a relevant first instance, we choose as ambient manifold the \emph{$n$-th dimensional Riemannian Heisenberg group}.  The Heisenberg group $\hh^n$, for $n\geq 1$, is $\rr^{2n+1}$ endowed with 
the non-Abelian group law
\[
(x,y,t)\ast(x,y',t')=\left(x+x',y+y',t+t'\sum_{j=1}^n\left(x_j'y_j-x_jy_j'\right)\right),
\]
where $x=(x_1,\ldots,x_n)$, $x'=(x_1',\ldots,x_n')$, $y=(y_1,\ldots,y_n)$ and $y'=(y_1',\ldots,y_n')$, which realizes it as (the most relevant instance of) \emph{Carnot group} (cf. \cite{MR2363343}). The Lie algebra of $\hh^n$ is generated by the family of left-invariant vector fields
\[
X_i=\dfrac{\partial}{\partial x_i} + y_i \dfrac{\partial}{\partial t}, \qquad Y_i=\dfrac{\partial}{\partial y_i} - x_i \dfrac{\partial}{\partial t}, \qquad T=\dfrac{\partial}{\partial t},
\]
for $i=1,\ldots,n$. A remarkable class of Riemannian structures can be defined on $\hh^n$ by choosing, for any $\eps\neq 0$, the unique Riemannian metric $g_\eps$ which makes  $\{X_1,\ldots,X_n,Y_1,\ldots,Y_n,\eps T\}$ an orthonormal frame. The importance of $(\hh^n,g_\eps)$ in the Riemannian framework is supported by several reasons. For instance, it appears in the classification of homogeneous $3$-spaces with isometry group of dimension $4$, usually denoted by $Nil_3(-1/\tau)$ (cf. \cite{MR2195187}).  When $\eps$ goes to $ 0$, the space $(\mathbb{H}^n,g_\eps)$ converges in Gromov-Hausdorff sense to the sub-Riemannian Heisenberg group $(\hh^n,\langle\cdot,\cdot\rangle)$, where $\langle\cdot,\cdot\rangle$ is the restriction of any of the metric $g_\eps$ to the \emph{horizontal distribution} $\hhh=\spann\{X_1,\ldots,X_n,Y_1,\ldots, Y_n\}$. The sub-Riemannian Heisenberg group $\hn$ is itself of fundamental importance in various settings, since it constitutes the prototypical model in the context of Carnot groups, sub-Riemannian manifolds (cf. \cite{MR3971262}) and \emph{CR manifolds} (cf. \cite{MR2312336}). In the following, we deal with a relevant class of non-parametric hypersurfaces, namely that of \emph{vertical or ($t$-) graphs}, i.e. Euclidean graphs over the horizontal hyperplane $\{t=0\}\equiv\rr^{2n}$. The equation of prescribed mean curvature of a $t$-graph over a bounded domain $\Om\subseteq \rr^{2n}$ for a given source $H$ is formally given by
\begin{equation}\tag{$\eps$-PMC}\label{pmc'}
\divv\left(\frac{Du+X}{\sqrt{\eps^2+|Du+X|^2}}\right)=H,
\end{equation}
where $X(x,y)=(-y,x)$ (cf. \cite{MR2262784,pmc1}). Our main achievement, in the spirit of \cite{MR487722}, is an existence and regularity result for solutions to \eqref{pmc'}, both in the non-extremal and in the extremal case.
\begin{theorem}\label{thm:1}
    Let $\Om\subseteq\rr^{2n} $ be a bounded domain with Lipschitz boundary, and let $H\in\lip(\Om)\cap C^{1,\gamma}_{loc}(\Om)$ for some $\gamma\in(0,1)$. Then \eqref{Giusti} holds if and only if there exists $u\in C^2(\Om)$ which is a classical solution to \eqref{pmc'} on $\Om$. Moreover, if $H\in C^{k,\gamma}_{loc}(\Om)$ for some $k\in\mathbb N$, $k\geq 1$, then $u\in C^{k+2,\gamma}_{loc}(\Om)$. Finally, if $H\in C^\infty(\Om)$, then $u\in C^\infty(\Om)$.
\end{theorem}
Our approach can be summarized in the following major points.
\begin{itemize}
    \item Following Giusti's scheme, we will first prove existence of $BV$ minimizers of a suitable penalized functional, analogous to \eqref{penafun}, in the non-extremal case (cf. \Cref{existenceeps}). To improve the regularity of such minimizers, we first rely on suitable variational properties of minimizers (cf. \Cref{minimplieslambdamin} and \Cref{prop:voldensity}) to infer that 
minimizers are locally bounded in $\Om$ (cf. \Cref{linftyloc}). 
\item Our second step relies on a generalization to the anisotropic setting (cf. \cite{MR4430590}) of some celebrated regularity results for almost-minimizers of the perimeter (cf. e.g. \cite{MR0179651,MR0171198,MR0420406,MR0467476,MR0667448,MR0648941}, and cf. \cite{MR4400905,LV2024} and references therein for recent developments).
Exploiting some results from \cite{MR4430590}, we shall see that the boundary $\ptl E_u$ of the subgraph of a minimizer $u$ is regular outside a singular portion with small Hausdorff dimension (cf. \Cref{reglambdamin}).
A crucial result then consists in translating these regularity properties from $\partial E_u$ to $u$. More precisely, we show that $u$ is regular outside a small set  $\Om_{u,0}\subseteq \Om$ for which $\mathcal{H}^s(\Om_{u,0})=0$ for any $s>2n-7$ (cf \Cref{prop:regu}). The proof of these results is based on a careful analysis of the prescribed mean curvature equation for \emph{intrinsic graphs} in the sense of \cite{MR2223801}. 
\item Next, owing to some structure properties of the Riemannian perimeter induced by the metric $g_\eps$ (cf. \Cref{reprlemma}), we show that minimizers enjoy Sobolev regularity (cf. \Cref{wunoloc}).
\item In view of the previous steps, and exploiting a suitable existence result for the Dirichlet problem in small balls (cf. \Cref{existence*}), we provide local Lipschitz regularity by means of an approximation procedure and a comparison principle argument (cf. \Cref{liploc}).
\item To pass from Lipschitz regularity to higher regularity, exploiting a well-established approach,
we write a linear uniformly elliptic equation for the function
\[
u_v(z)=\frac{u(z+v)-u(z)}{|v|},
\]
so that both the De Giorgi-Nash-Moser theory for $C^{1,\alpha}$ regularity and the classical Schauder theory for higher regularity apply (cf. \Cref{uisc2}). This last step basically concludes the proof of \Cref{thm:1} in the non-extremal case.
\item  Finally, the existence of classical solutions in the extremal case (cf. \Cref{extremalcases}) follows exploiting the aforementioned approximating procedure, together with a suitable compactness argument and the extension to the Riemannian setting of the Euclidean notion of generalized solution to \eqref{PMC}.
\end{itemize}

As already mentioned, since our source $H$ may not be constant, a crucial step in the proof of local Lipschitz regularity is the use of suitable interior gradient estimates. More precisely, we extend the proof of Korevaar and Simon (cf. \cite{MR0932680}) to achieve the following result, which may be of independent interest.
\begin{theorem}[Interior gradient estimates]\label{mainige'}
    Let $\Om\subseteq\rr^{2n}$ be a bounded domain. Let $H\in C^1(\Om)$. Let $u\in C^3(\Om)$ be a solution to \eqref{pmc'} and let $\tilde\Om\Subset\Om$ be a domain. For any  domain $\hat\Om\Subset\tilde\Om$ there exists a constant 
    $C=C\left(n, \eps,d(\partial\hat\Om,\partial\tilde\Om),\|u\|_{L^\infty(\tilde\Om)},\|H\|_{C^1(\tilde\Om)},\|X\|_{L^\infty(\tilde\Om)}\right)>0$
    such that
    \begin{equation}\label{ige2024'}
    \|Du\|_{L^\infty(\hat\Om)}\leq C,
    \end{equation}
    where $d(\partial\hat\Om,\partial\tilde\Om)$ the Euclidean distance between $\partial\hat\Om$ and $\partial\tilde\Om$.
\end{theorem}
 In the proof of \Cref{mainige'}, a crucial role is played by the identity
  \begin{equation}\label{precodige}
        \begin{split}
\Delta_S (\vet)&=g_\eps(\nabla_S H, \eps T) -\ve_{2n+1}\left(\ric_\eps(\ve,\ve)+|h^\eps|^2\right),
        \end{split}
    \end{equation}
   which holds for an embedded orientable hypersurface of class $C^3$.
Exploiting \eqref{precodige}, we can study the sign of the Laplace-Beltrami operator $\Delta_S$ associated to $S=\text{graph}(u)$, applied to the vertical coordinate of the Riemannian normal to the graph, $\nu^\eps_{2n+1}$.  
In \eqref{precodige}, $\nabla_S$, $\ric_\eps$ and $h^\eps$ stand for the gradient in $S$, the Ricci curvature of $(\mathbb{H}^{2n+1},g_\eps)$ and the second fundamental form of $S$ respectively. Remarkably, assuming the additional condition
 \begin{equation}\label{Serrinstrict}
|H|< H_{\ptl\Om}(z_0)
\end{equation}
for any $z_0\in \ptl\Om$, i.e. the sub-optimal version of \eqref{Serrin} required in \cite{MR2332426,pmc1}, the proof of \Cref{mainige'} can be adapted to provide global gradient estimates (cf. \Cref{mainigeglobaldue}). Consequently,  we can improve the existence results proved in \cite{MR2332426,pmc1} for the Riemannian Dirichlet problem associated to \eqref{pmc'}. More precisely, we generalize \cite[Theorem 6.1]{pmc1} to the case in which $H$ is not constant. 
 \begin{theorem}\label{existence*}
     Let $\Om\subseteq \rr^{2n}$ be a bounded domain with $C^{2,\alpha}$ boundary, for some $\alpha\in (0,1)$. Let $\varphi \in C^{2,\alpha}(\overline\Om)$ and let $H\in\lip(\Om)$.
     Assume that \eqref{Giusti} holds and that $\Om$ is a non-extremal domain. Assume in addition that \eqref{Serrinstrict} holds.
     Then, for any $\eps\neq 0$ there exists $u_\eps\in \lip(\Om)\cap  C^{2,\alpha}_{loc}(\Om)$ which solves \eqref{pmc'} on $\Om$ and such that $u_\eps=\varphi$ on $\partial \Om$. If in addition $H\in C^{1,\alpha}(\overline\Om)$, then $u_\eps\in C^{2,\alpha}(\overline\Om)$.
\end{theorem}
Regarding uniqueness, we shall prove that the extremal condition \eqref{Giusti2} is equivalent to the maximality of the domain and the verticality of solutions to \eqref{pmc'}, in a weak sense, at the boundary. Moreover, we show that and any of these conditions implies uniqueness of solutions of \eqref{pmc'} up to vertical translations.
\begin{theorem}\label{thm:uniqueness}
    Let $\Om\subseteq\rr^{2n}$ be a bounded domain with Lipschitz boundary. Let $H\in\lip(\Om)\cap C^{1,\gamma}_{loc}(\Om)$ for some $\gamma\in(0,1)$ satisfy \eqref{Giusti}. The following statements are equivalent.
    \begin{itemize}
        \item [(i)] If $\hat\Om\subseteq\rr^{2n}$ is any domain such that $\Om\subsetneq\hat\Om$, then there is no solution $u\in C^2(\hat\Om)$ to \eqref{pmc'} in $\hat\Om$.
        \item [(ii)] \eqref{Giusti2} holds.
        \end{itemize}
        Moreover, if $\partial\Om$ is of class $C^2$, then $(i)$ and $(ii)$ are equivalent to the following condition.
        \begin{itemize}
        \item [(iii)] For any $u\in C^2(\Om)$ which solves \eqref{pmc'} in $\Om$, it holds that
        \begin{equation*}
            \lim_{t\to 0^+}\left|\int_{\partial \Om_t}\frac{\langle\nu_t,Du+X\rangle}{\sqrt{\eps^2+|Du+X|^2}}\,d\mathcal{H}^{2n-1}\right|=\mathcal{H}^{2n-1}(\partial\Om),
        \end{equation*}
        where $\Om_t=\{z\in\Om\,:\,\min_{w\in\partial\Om}|z-w|>t\}$ is defined for $t>0$ small enough and $\nu_t$ is its exterior unit normal.
    \end{itemize}
    Finally, if $\partial\Om$ is of class $C^2$, each of the previous three conditions implies that solutions to \eqref{pmc'} in $\Om$ are unique up to vertical translations.
\end{theorem}
With regard to the sub-Riemannian Heisenberg group $\hh^n$, it is well known (cf. \cite{MR3510691,MR2262784,pmc1}) that the equation of prescribed mean curvature of a $t$-graph over $\Om\subseteq \rr^{2n}$ is formally given by
\begin{equation}\tag{$\hhh$-PMC}\label{pmchor}
\divv\left(\frac{Du+X}{|Du+X|}\right)=H.
\end{equation}
Due to the possible presence of \emph{characteristic points}, i.e. points in $\Om$ where $Du+X$ vanishes, \eqref{pmchor} may be both degenerate elliptic and singular (cf. \cite{MR2262784,pmc1}). Therefore, this issue should be understood in a weak sense. More precisely, the problem of finding $t$-graphs with prescribed mean curvature can be formulated by looking for local minimizers of the functional
\begin{equation}\label{subfunintro}
\mathcal{I}(u)=\int_\Om |Du+X|+\int_\Om Hu \,dz.
\end{equation}
When $H=0$, the regularity of such minimizers has carried the attention of many authors. In the papers of Ritoré (cf. \cite{MR2448649}) and Serra Cassano and Vittone (cf. \cite{MR3276118}) there are examples in $\hh^1$ both of entire area-minimizing $t$-graphs with merely locally Lipschitz regularity and of even discontinuous area-minimizing $t$-graph.
Assuming that $H=0$ and that minimizers have at least $C^1$ regularity,  Cheng, Hwang and Yang (cf. \cite{MR2481053}) proved the presence of a foliation by $C^2$ curves in the non-characteristic set. Later on, Ritoré and Galli (cf. \cite{MR3412382}) extended this result for a continuous source $H$,  and Ritoré and Giovannardi (cf. \cite{MR4314055}) obtained the same property under the weaker assumption of Lipschitz regularity and intrinsic $C^1$ regularity in the sub-Finsler Heisenberg group. Regarding the existence of such solutions, it is customary to rely on the aforementioned Riemannian approximation (cf. \cite{MR2262784,MR2043961}). A  similar approximation scheme was considered in the sub-Riemannian setting in \cite{MR2774306,MR2583494} to study the Lipschitz regularity for non-characteristic minimal surfaces.  For a detailed analysis of this approach, we refer to \cite{MR3510691}. Accordingly, our existence result can be applied to study minimization problems related to \eqref{subfunintro}. More precisely, we obtain solutions in $BV_{loc}(\Om)\cap L^\infty_{loc}(\Om)$ to the sub-Riemannian prescribed mean curvature equation in the following sense.
\begin{theorem}\label{thm:sub-Riem}
    Let $\Om\subseteq\rr^{2n}$ be an open and bounded set with Lipschitz boundary. Let $H\in\lip(\Om).$ Assume that $\eqref{Giusti}$ holds.
    Then, there exists $u\in BV_{loc}(\Om)\cap L^\infty_{loc}(\Om)$ such that $u$ is an $H$-minimizer for $P_\hh$ on $\Om\times\rr$ in the sense of \eqref{limminloc2}. Moreover,  there exist a sequence of open sets such that $ \Om_j\Subset\Om_{k}\Subset\Om$  for any $j<k$ and $\bigcup_{j=0}^\infty\Om_j=\Om$ and a sequence $(u_j)_j\subseteq C^\infty(\Om_j)$, such that each $u_j$ solves \eqref{pmc'} in $\Om_j$ and moreover
    \begin{equation*}
        u_j\to u\text{ almost everywhere in $\Om$ }\qquad\text{and}\qquad P_{\varepsilon_j}(E_{u_j},\cdot)\rightharpoonup^*P_{\mathbb H}(E_u,\cdot)\text{ locally in }\Om\times\rr,
    \end{equation*}   
 where $\rightharpoonup^*$ denotes the \emph{weak-$^*$ convergence} of measures.
\end{theorem}
We point out that \Cref{thm:sub-Riem} generalizes the existence result proved in \cite{MR3276118} for minimal $t$-graphs allowing $H$ to be different from zero. Another interesting difference with respect to \cite{MR3276118} consists in the different approach to the existence issue. Indeed, Serra Cassano and Vittone provided existence essentially via direct methods. On the other hand, although we believe that the same strategy could have worked as well in our framework, we preferred to provide existence combining direct methods and approximation. Indeed, as nicely explained in \cite{MR2774306}, sub-Riemannian minimizers arising as limit of Riemannian minimizers are only a particular subfamily of all sub-Riemannian minimizers, and they typically enjoy better regularity properties.
Finally, as a byproduct of the computations done to prove \eqref{precodige}, we recover in \Cref{thm:2v} the interpretation of the second variation formula of a smooth non-characteristic hypersurface in the sub-Riemannian Heisenberg group $\hh^n$ as a Riemannian limit stated in \cite{MR2710215,MR2723818}.
\\

This paper is organized as follows. In \Cref{sec:preliminaries}, we fix the notation and we provide some background, focusing mainly on some properties of the perimeter of a subgraph. In \Cref{sec:geomprop} we provide some useful tools in order to deal with embedded hypersurfaces in the Riemannian Heisenberg setting. \Cref{sec:ige} is devoted to the proof of \Cref{mainige'} and \Cref{mainigeglobaldue}, i.e. interior and global gradient estimates. \Cref{sec:mainsection} contains the existence and regularity results for solutions to \eqref{pmc'}, namely \Cref{thm:1} and \Cref{existence*}. In \Cref{sec:uniqueness} we prove the uniqueness result \Cref{thm:uniqueness}. In \Cref{applisub} we prove \Cref{thm:sub-Riem} and \Cref{thm:2v}.
Finally, in \Cref{sec:proofs} we prove \eqref{precodige}, including the proofs of some facts stated in \Cref{sec:preliminaries} and \Cref{sec:geomprop}.

\section{Preliminaries}\label{sec:preliminaries}

\subsection{Notation}\label{notsec} Unless otherwise specified, we let $n\in\mathbb{N}$, $n\geq 1$ and we denote by $\Om$ a bounded domain in $\rr^{2n}$ with Lipschitz boundary. Given a measurable set $A\subseteq \rr^n$, we denote by $\overline A$ the closure of $A$ and by $\chi_A$ the characteristic function of $A$. Given two open sets $A,B\subseteq \rr^n$, we write $A\Subset B$ whenever $\overline A\subseteq B$. The $n$-dimensional Lebesgue measure in $\rr^n$ is denoted by $|\cdot|$. The $s$-dimensional Hausdorff measure with respect to the Euclidean distance is denoted by $\mathcal{H}^s$. The Euclidean ball in $\rr^n$ centered at $p\in\rr^n$ with radius $r>0$ is denoted by $B(p,r)$.

\subsection{The Heisenberg group $\hh^n$}

We follow the notation and background given in \cite{MR2435652}. We define a non-Abelian group law $\cdot$ in $\rr^{2n+1}$ by 
\begin{equation*}
    p\cdot p'=( x, y,t)\cdot (x', y',t')=\left( x+x', y+ y', t+t'+\sum_{j=1}^n\left(x_j'y_j-x_jy_j'\right)\right),
    \end{equation*}
where we denote points $p\in\mathbb R^{2n+1}$ by $p=(z,t)=(x, y,t)=(x_1,\ldots,x_n,y_1,\ldots,y_n,t)$.
The Lie group $(\rr^{2n+1},\cdot)$ is referred to as the \emph{$n$-th Heisenberg group} and is denoted by $\hh^n$. For any $p\in\hh^n$, the left-translation by $p$ is the diffeomorphism $\ell_p(q)=p\cdot q$.  A basis of left-invariant vector fields is given by
\begin{equation*}
X_i=\frac{\ptl}{\ptl x_i}+y_i\,\frac{\ptl}{\ptl t}, \qquad
Y_i=\frac{\ptl}{\ptl y_i}-x_i\,\frac{\ptl}{\ptl t}, \qquad
T=\frac{\ptl}{\ptl t},
\end{equation*}
where $i=1\ldots,n$.
The horizontal distribution $\mathcal{H}$ in $\hh^n$ is the smooth  distribution generated by $X_1,Y_1,\ldots,X_n,Y_n$.  A vector field $U$ is said \emph{horizontal} if $U(p)\in\mathcal H_p$ for any $p\in\hh^n$. Given a smooth function $f$, we define the \emph{horizontal gradient} of $f$ by $$\nabla^\hh f=\sum_{i=1}^n X_i(f)X_i+Y_i(f)Y_i.$$
 For any $\eps\neq 0$, the 
 matrix of change of basis from $X_1,\ldots,Y_n,\varepsilon T$ to the canonical basis of $T\rr^{2n+1}$ at $p=(x_1\ldots x_n,y_1,\ldots,y_n,t)$ is given by
 \begin{equation}\label{base}
 C_\varepsilon=
\begin{pmatrix}
  1 &  \ldots & 0 & 0 & \ldots & 0 & 0 \\
  \vdots & \ddots & \vdots & \vdots & \ddots  & \vdots & \vdots \\
  0 &  \ldots & 1 & 0 & \ldots & 0 & 0 \\
  0 &  \ldots & 0 & 1 & \ldots & 0 & 0 \\
    \vdots & \ddots & \vdots & \vdots & \ddots  & \vdots & \vdots \\
  0 &  \ldots & 0 & 0 & \ldots & 1 & 0 \\
  y_1 &  \ldots & y_n & -x_1 & \ldots & -x_n & \eps \\
\end{pmatrix}.
 \end{equation}
In the following, we consider for any $\eps\neq 0$ the left-invariant Riemannian metric $g_\eps=\escpr{\cdot\,, \cdot}_\eps$ on $\hh^n$ such that $\{X_1,Y_1,\ldots,X_n,Y_n,\eps T\}$ is an orthonormal basis at every point. We shall drop the subindex in the metric $\escpr{\cdot,\cdot}$ when considering horizontal vectors. We shall also use the compact notation $\{Z_1,\ldots Z_{2n+1}\}$ to denote the previous basis. Moreover, we denote by $\nabla^\eps$ its associated Levi-Civita connection. The norm of a vector field $U$ with respect to $g_\eps$ will be denoted by $|U|_\eps$. 
 By means of the \emph{Koszul formula} (cf. \cite{MR1138207}), the following relations hold.
\begin{alignat}{2}
&\notag \nabla_{X_i}^\eps X_j=0, \qquad \qquad  \ \, \nabla^\eps_{Y_i}Y_j=0, \qquad \ \,\nabla^\eps_{T}T=0, \\
\label{levi-civita}&\nabla^\eps_{X_i}Y_j=-\delta_{i,j}T, \qquad \, \nabla^\eps_{X_i}\eps T=\frac{Y_i}{\eps}, \quad \ \ \nabla^\eps_{Y_i}\eps T=-\frac{X_i}{\eps}, \\
&\notag \nabla^\eps_{Y_i}X_j=\delta_{i,j}T, \qquad \ \ \,\,\nabla^\eps_{\eps T}X_i=\frac{Y_i}{\eps}, \  \ \quad \nabla^\eps_{\eps T}Y_i=-\frac{X_i}{\eps}
\end{alignat}
for any $1\leq i,j\leq n$, where $\delta_{i,j}$ is the Kronecker delta.
Setting $J(U)=\nabla^1_U T$ for any vector field $U$, we can easily deduce that $J(X_i)=Y_i$, $J(Y_i)=-X_i$ and $J(T)=0$.
 The Riemannian volume of a set $E$ for $g_\eps$ is, up to a constant, the Lebesgue measure in $\rr^{2n+1}$. 
Given a norm $|\cdot|$ in $T\hh^n$, and $v\in T_p\hh^n$, we will drop the subindex $p$ and write $|v|$. 
We denote by $\ric$ the quadratic form associated with the Ricci tensor induced by $g_\eps$, that is
    \begin{equation*}
    \begin{split}
\ric(U)&=\sum_{j=1}^{2n+1}g_\varepsilon\left(\n{Z_j}\n{U}U-\n{U}\n{Z_j}U+\n{[U,Z_j]}U,Z_j\right)
    \end{split}
     \end{equation*}
for any vector field $U$ of class $C^2$.
As we will show in \Cref{sec:proofs},
    \begin{equation}\label{ricciexpress}
        \ric(U)=-\frac{2}{\eps^2}\g{U,U}+(2n+2)\frac{u_{2n+1}^2}{\eps^2}
    \end{equation}
 for any vector field $U=\sum_{j=1}^{2n+1}u_jZ_j\in C^2(\mathbb H^n;T\mathbb H^n)$.
\subsection{Carnot-Carathéodory structure on $\hn$}
If $\Gamma:[a,b]\scu \hn$ is an absolutely continuous curve, we say that it is \emph{horizontal} whenever $ \Dot\Gamma(t)\in\hhh_{\Gamma(t)}$
for almost every $t\in [a,b]$,
and we say that it is \emph{sub-unit} whenever it is horizontal with $|\Dot\Gamma(t)|=1$ for almost every $t\in[a,b]$.
Moreover, we define 
\begin{equation}\label{dccaprile}
    d^\hh(p,p'):=\inf\{T\,:\,\Gamma:[0,T]\scu \mathbb{H}^n\text{ is sub-unit, $\Gamma(0)=p$ and $\Gamma(T)=p'$}\}
\end{equation}
which, by the Chow-Rashevskii theorem (cf. \cite{MR0001880}), defines a distance on $\mathbb{H}^n$, called \emph{Carnot-Carathéodory distance}. The metric space $(\hn,d^\hh)$ is then a prototype of \emph{Carnot-Carathéodory space} (cf. \cite{MR1421823}).

\subsection{Embedded hypersurfaces in $\hh^n$}
We consider oriented hypersurfaces of class $C^1$ embedded in $\hh^n$, and we choose a unit normal to $S$. In case $S$ is the boundary of a domain in $\hh^n$, we always choose the outer unit normal $\ve$. The \emph{characteristic set} of $S$, denoted by $S_0$, is the set of points $p\in S$ where the tangent space $T_pS$ coincides with the horizontal distribution $\hhh_p$. The \emph{horizontal unit normal} $\vh$ is defined in $S\setminus S_0$ by
\[
\vh=\frac{\ve_h}{|\ve_h|},
\]
where we denote by $U_{h}$ the orthogonal projection of a vector $U$ onto $\mathcal{H}$.


\subsection{Perimeters in $\hh^n$}

In this subsection, $\Om\subseteq \rr^{2n}$ is a bounded open set and $\eps\neq 0$. We denote the Euclidean norm, total variation,  perimeter and divergence in $\rr^{2n+1}$ by $|\cdot|$, $Var$, $P$ and $\divv_{eu}$ respectively, and we denote the space of functions of bounded variation by $BV_{eu}$.
Given an open set $A\subseteq\hh^n$ and $f\in L^1(A)$, the \emph{$\eps$-variation} of $f$ in $A$ is denoted by
\[
Var_\eps(f;A)=\sup\left\lbrace \int_{A} f\divv_\eps (U)\, dx: U\in C_c^1(A;T\hh^n), \ |U|_{\eps,\infty}\leq 1\right\rbrace,
\]
where $C_c^1(A;T\hh^n)$ is the space of $C^1$ compactly supported vector fields in $A$, $|U|_{\eps,\infty}=\sup_{p\in A}|U(p)|_{\eps}$ and $\divv_\eps$ is the divergence associated to $g_\eps$. The space of $L^1(A)$ functions with bounded $\eps$-variation is denoted by $BV_\eps(A)$.
Given $E\subseteq\hh^n$ a measurable set and $A\subseteq\hh^n$ an open set, the $\eps$-perimeter of $E$ in $A$ is given by
$$
P_\eps(E;A)=Var_\eps (\chi_E;A).
$$
In case $A=\hh^n$ we write $P_\eps(E;\hh^n)=P_\eps(E)$. Given a vector field $U\in C_c^1(A;T\hh^n)$, it follows from \eqref{base} that
\begin{equation}\label{equaldiv}
    \divv_{eu} U=\divv_\eps U,
    \end{equation} so that from now on we shall simply write $\divv$.
There exist constants $C(\Om,\eps)>0$ and $c=c(\Om,\varepsilon)>0$ such that
      \begin{equation}\label{eq:normequiv}
       C|v| \leq |v|_{\varepsilon}\leq c|v|
       \end{equation}
for any $v\in T_p(\Om\times\rr)$. Indeed, we assume for clarity that $n=1$ and let $v=v_1\frac{\ptl}{\ptl x}+v_2\frac{\ptl}{\ptl y}+v_3\frac{\ptl}{\ptl t}$. Using \eqref{base} and recalling that $(a+b)^2\leq 2a^2+2b^2$ for any $a,b\in\rr$, we get
\begin{equation*}
    \begin{split}
|v|_{\varepsilon}^2
\leq v_1^2+v_2^2+\frac{2}{\varepsilon^2}v_3^2+\frac{4}{\varepsilon^2}y^2v_1^2+\frac{4}{\varepsilon^2}x^2v_2^2
\leq 4\left(1+\frac{1+\max_{z\in\bar{\Om}}|z|^2}{\eps^2}\right)|v|^2 \color{black},
    \end{split}
\end{equation*}
 On the other hand, 
\begin{equation*}
    \begin{split}
 |v|^2 
 \leq v_1^2+v_2^2+4x^2v_2^2+4y^2v_1^2+2(v_3-yv_1+xv_2)^2
 \leq \max\left\{1+4\max_{z\in\overline\Om}|z|^2,2\eps^2\right\}|v|_{\varepsilon}^2.\color{black}
    \end{split}
\end{equation*}
   Notice that, in \eqref{eq:normequiv}, $c$ can be chosen uniformly in $\eps$ for $|\eps|$ big enough, while $C$ can be chosen uniformly in $\eps$ for $|\eps|$ small enough.
  Given an open set $A\subseteq\Om\times\rr$  and a vector field $U\in C_c^1(A;T\hh^n)$ with $|U|_{\eps,\infty}\leq 1$, it follows from \eqref{eq:normequiv} that $|CU|\leq |U|_\eps\leq1$
    and 
    $$
    \int_A f\divv(U)dx=\frac{1}{C}\int_A f\divv(CU)dx \leq \frac{1}{C}Var(f,A)
    $$
    for any $f\in L^1(\Om\times \rr)$. Hence $Var_{\eps}(f,A)\leq \frac{1}{C}Var(f,A)$. Similarly, $Var(f,A)\leq \frac{1}{c} Var_{\eps}(f,A)$, so that
    \begin{equation}\label{ineq:perieq}
    \frac{1}{c} Var(f,A)\leq Var_\varepsilon(f,A)\leq \frac{1}{C}Var(f,A).
    \end{equation}
Notice that \eqref{ineq:perieq} implies that $BV_\eps(A)$ and $BV_{eu}(A)$ coincide  for $A\subseteq\Om\times\rr$, and shall be denoted by $BV(A)$. Moreover, the perimeters $P$ and $P_\eps$ are absolutely continuous with respect to each other. Hence, the Euclidean reduced and essential boundaries of a Caccioppoli set coincide with the ones induced by $P_\varepsilon$. In the following, we will denote by $\partial^* E$ the reduced boundary of $E$. 
Given $E\subseteq \hh^n$ measurable and $A\subseteq \hh^n$ open, the \emph{horizontal perimeter} of $E$ in $A$ is defined by 
\begin{equation*}
    P_{\mathbb H}(E,A)=\sup\left\{\int_E \divv U\,dx,\,U\in C^1_c(A,\mathcal{H}),\,|U|_{1,\infty}\leq 1\right\},
\end{equation*}
where $C^1_c(A,\mathcal{H})$ is the space of $C^1$ compactly supported horizontal vector fields in $A$. We refer to \cite{MR1871966} for the main properties of the horizontal perimeter.
We point out that both $P_\eps$ and $P_\hh$ behave like the Euclidean perimeter $P$ for vertical sets. More precisely, given a Caccioppoli set $E\subseteq\rr^{2n}$ and an open set $A\subseteq \mathbb H^n$, arguing as in \cite[(3.2)]{MR1312686} and observing that the last component of the measure theoretic Euclidean unit normal to $E\times\rr$ is zero, then
\begin{equation}\label{trace}
P_\eps(E\times\rr;A)=P_\hh(E\times\rr;A)=P(E\times\rr,A).
\end{equation}
By the above definitions, the following relations between horizontal and $\eps$-perimeter hold.


\begin{proposition}\label{pepestoph}
    Let $A\subseteq \hh^n$ be an open bounded set and $F\subseteq \hh^n$ be a Caccioppoli set. Then 
    \begin{equation*}
        P_{\mathbb H}(F,A)\leq P_\varepsilon(F,A)\qquad\text{and}\qquad
        \lim_{\varepsilon\to 0}P_{\varepsilon}(F,A)=P_{\mathbb H}(F,A).
    \end{equation*}
    Moreover, if $(\varepsilon_j)_j\subseteq(0,1)$ satisfies $\varepsilon_j\searrow 0$ as $j\to\infty$, and $E$ and $(E_j)_j$ are measurable sets such that $\chi_{E_j}\to\chi_E$ in $L^1_{loc}(A)$, then
    \begin{equation}\label{doublelsc}
P_\hh(E;A)\leq\liminf_{j\to\infty}P_{\eps_j}(E_j;A).
    \end{equation}
\end{proposition}


Moreover, the following compactness result holds.
\begin{proposition}\label{firstcompact}
    Let $A\subseteq\mathbb H^n$ be an open and $(E_k)_k$ be a sequence of finite $\mathbb H$-perimeter sets in $A$. Assume that there exists $M>0$ such that 
    $\sup_kP_{\mathbb H}(E_k,A)<M$.
    Then there exists a finite $\mathbb H$-perimeter set $E$ in $A$ such that $\chi_{E_k}\to\chi_E$ in $L^1_{loc}(A)$.
\end{proposition}

\begin{proof}
    Let $A'\Subset A$ open. Let $B_1,\ldots,B_k$ be a covering of $A'$ of Carnot-Carathéodory balls, i.e. the metric balls with respect to the distance introduced in \eqref{dccaprile}, such that
    \begin{equation*}
        A'\subseteq\bigcup_{j=1}^sB_j\Subset A.
    \end{equation*}
    Notice that $\chi_{E_k}\in BV_{\mathbb H}(B_j)$ for any $j=1,\ldots,k$ (cf. \cite{MR1437714} for the definition of $BV_{\mathbb H}$). Let us consider first $B_1$. In view of \cite[Theorem 2.2.2]{MR1437714}, there exists $v_k\in C^\infty(B_1)\cap BV_{\mathbb H}(B_1)$ such that
    \begin{equation*}
        \|v_k-\chi_{E_k}\|_{L^1(B_1)}\leq \frac{1}{k}\qquad\text{and}\qquad\left|P_{\mathbb H}(E_k,B_1)-\int_{B_1}|\nabla^\hh v_k|\,dx\right|\leq\frac{1}{k}
    \end{equation*}
    Therefore
    \begin{equation*}
        \|v_k\|_{L^1(B_1)}\leq |E_k\cap B_1|+1\qquad\text{and}\qquad
        \int_{B_1}|\nabla^\hh v_k|\,dx\leq \sup_{k}P_{\mathbb H}(E_k,A)+1,
    \end{equation*}
    so that $(v_k)_k$ is bounded in $BV_{\mathbb H}(B_1)$. Hence, \cite{MR1404326} implies that there exists $v\in BV_{\mathbb H}(B_1) $ such that, up to a subsequence, $v_k\to v$ in $L^1(B_1)$. This fact trivially implies the existence of a set $E^1\subseteq B_1$ such that, up to a subsequence, $\chi_{E_k}\to\chi _{E^1}$ in $L^1(B_1)$. The thesis then easily follows by a diagonal process and the lower semicontinuity of the $\mathbb H$-perimeter. \qedhere
    
\end{proof}

\subsection{$t$-graphs in $\hh^n$}

Given a bounded open set $\Om\subseteq\rr^{2n}$ and a measurable function $u:\Om\longrightarrow[-\infty,+\infty]$, we write the subgraph of $u$ as  
\begin{equation}\label{def:subgraph}
E_u=\{(z,t)\in\Om\times\rr\,:\,t<u(z)\}.
\end{equation}
A simple computation shows that, for $u\in W^{1,1}_{loc}(\Om)$ and $\tilde\Om\subseteq \Om$ open, the perimeter of $E_u$ in $\tilde\Om\times\rr$ can be computed as
\begin{equation}\label{periforw11}
    \mathcal A_\varepsilon(u,\tilde\Om)=\int_{\tilde\Om}\sqrt{\varepsilon^2+|D u+X|^2}\,dz,
\end{equation}
 where $X:\rr^{2n}\to\rr^{2n}$ is defined by $X(x,y)=(-y,x)$ and $Du$ is the gradient of $u$. The $L^1$-relaxation of $\mathcal A_\varepsilon$ for $u\in BV (\Om)$ is
 \begin{equation*}
     \overline{\mathcal A}_\varepsilon(u,\tilde\Om)=\inf\left\{\liminf_{k\to\infty}\mathcal{A}_\varepsilon(u_k,\tilde\Om):\,(u_k)_k\subseteq W^{1,1}(\Om),\,u_k\to u\text{ in }L^1(\tilde\Om) \right\}.
 \end{equation*}
We also define
\begin{equation*}
    S_\varepsilon (u,\tilde\Om)=\sup\left\{\int_{\tilde\Om}(-u\divv\tilde g+\langle X,\tilde g\rangle+\varepsilon g_{2n+1})\,dz:\,g=(\tilde g,g_{2n+1})\in C^1_c(\tilde\Om,\rr^{2n+1}),\,|g|\leq 1\right\}.
\end{equation*}
\begin{lemma}\label{equifun}
Given a bounded open set $\Om\subseteq\rr^{2n}$, $\tilde\Om\subseteq \Om$ open and $u\in L^1(\Om)$, then
    $$
P_\varepsilon(E_u,\tilde\Om\times\rr)=\overline{\mathcal A}_\varepsilon(u,\tilde\Om)=S_\varepsilon(u,\tilde\Om).
    $$
\end{lemma}
\begin{proof}
    The proof follows exactly as the proof of \cite[Theorem 3.2]{MR3276118}.
\end{proof}
Recall that for any $u\in BV$, its distributional derivative $\tilde Du$ can be decomposed as the sum of the two mutually singular measures $Du\mathcal{L}^{2n}+(Du)_s$, where $Du\in L^1$ 
and $(Du)_s$ is singular with respect to the Lebesgue measure $\mathcal{L}^{2n}$.
\begin{lemma}\label{reprlemma}
    Given a bounded open set $\Om\subseteq\rr^{2n}$, $\tilde\Om\subseteq \Om$ open and $u\in BV(\Om)$, it holds that
    \begin{equation*}
P_\varepsilon(E_u,\tilde\Om\times\rr)=(Du)_s(\tilde\Om)+\int_{\tilde\Om}\sqrt{\varepsilon^2+|Du+X|^2}\,dz.
    \end{equation*}
\end{lemma}
\begin{proof}
    Let us define $L:C^1_c(\Om,\rr^{2n+1})\longrightarrow \rr$ by
\begin{equation*}
    L(g)=\int_A(-u\divv \bar g+\langle X,\bar g\rangle+\varepsilon g_{2n+1})\,dz,
\end{equation*}
    where $g=(\bar g, g_{2n+1})$. $L$ is clearly linear. Moreover, since by \Cref{equifun} $S_\varepsilon(u)<+\infty,$ then $L$ extends to a linear bounded functional on $C^0_c(A,\rr^{2n+1})$. Therefore, by Riesz Theorem (cf. e.g. \cite[Theorem 1.54]{MR1857292}) there exists a unique $(2n+1)$-valued finite Radon measure $\mu$ such that
    \begin{equation*}
        L(g)=\int_A g\cdot d \mu\qquad\text{and}\qquad S_\varepsilon(u,A)=|\mu|(A)
    \end{equation*}
   for any $g\in C^1_c(A,\rr^{2n+1})$. By the uniqueness of such a measure it is easy to see that $\mu=(Du+Xd\mathcal{L}^{2n},\varepsilon d\mathcal{L}^{2n})$, and so 
   \begin{equation*}
       S_\varepsilon(u,A)=|(Du+Xd\mathcal{L}^{2n},\varepsilon d\mathcal{L}^{2n})|(A).
   \end{equation*}
   A trivial computation, together with \Cref{equifun}, concludes the proof.
\end{proof}

\section{Geometry of hypersurfaces in $(\hn,g_\eps)$}\label{sec:geomprop}
Along this section, we fix $\eps\neq 0$ and the metric $g_\eps$, and write as $\{Z_1,\ldots,Z_{2n},Z_{2n+1}\}$  the frame $\{X_1,\ldots, X_n,Y_1,\ldots,Y_n,\eps T\}$ to keep a compact notation. 
We let $S\subseteq\hh^n$ be an embedded orientable hypersurface of class $C^3$ with Riemannian unit normal $\ve$.
 We recall that
\begin{equation}\label{propv1}
    \sum_{j=1}^{2n+1}(Z_i\ve_j)\ve_j=0
\end{equation}
for any $i=1,\ldots,2n+1$ for any unitary extension of $\ve$, being $(\v_1,\ldots,\v_{2n+1})$ the coordinates of $\v$ related to $\{Z_1,\ldots,Z_{2n+1}\}$.
Moreover, if we denote by $d$ be the signed Riemannian distance from $S$,
%
%
then $d$ is of class $C^3$ near $S$, and satisfies the Eikonal equation $|\nabla d|=1$
 in a neighborhood of $S$, where we used the compact notation $\nabla=\nabla^\eps$. Therefore, $\ve$ can be extended to a suitable neighborhood of $S$ by letting $\ve=\nabla d$.
With this extension, we have
\begin{equation}\label{propv2}
    Z_i(\v_j)=Z_j(\v_i)
\end{equation}
 for any $i,j=1,\ldots,2n+1$ such that either $i=2n+1$, $j=2n+1$ or $|j-i|\neq n$.
  Finally,
\begin{equation}\label{propv3}
    X_i(\v_{n+i})=Y_i(\v_i)-\frac{2\vt}{\eps}\qquad\text{and}\qquad Y_i(\v_i)=X_i(\v_{n+i})+\frac{2\vt}{\eps}
\end{equation}
for any $i=1,\ldots,n$.
Hence
\begin{equation}\label{weird1}
    \sum_{j=1}^{2n+1}(Z_j\ve_i)\ve_j=-2\frac{\vet}{\eps}J(\ve)_i
\end{equation}
for any $i=1,\ldots,2n+1$, where we recall that $J(\ve)=(-\ve_{n+1},\ldots,-\ve_{2n},\ve_1,\ldots,\ve_n,0)$. 
We denote by $h^\eps$ and $H^\eps$ the associated \emph{second fundamental form} and \emph{mean curvature} of $S$ respectively, i.e., for a given $p\in S$, $h^\eps_p(v,w)=g_\eps(\nabla_{v}\ve,w)$
for any $v,w\in T_p S$, and
\begin{equation}\label{meancurvexpr}
    H^\eps(p)=\sum_{i=1}^{2n}h^\eps_p(e_i,e_i)=\sum_{i=1}^{2n}g_\eps( \nabla_{e_i}\ve,e_i)=\sum_{i=1}^{2n+1}Z_i\ve_i(p)
\end{equation}
for any orthonormal basis $e_1,\ldots, e_{2n}$ of $T_p S$. It is possible to express the norm of $h^\eps$ by
       \begin{equation}\label{riemsomma}
|h^\eps|^2=\sum_{l,s=1}^{2n+1}Z_s(\v_l)Z_l(\v_s)+4\left\langle J(\v),\nabla\left(\frac{\vt}{\eps}\right)\right\rangle+(2n-2)\frac{(\vt)^2}{\eps^2}+\frac{2}{\eps^2},
    \end{equation}
    so that, combining \eqref{ricciexpress} and \eqref{riemsomma}, we infer that
    \begin{equation}\label{needed inproposition}
       \ric(\ve)+ |h^\eps|^2=\sum_{l,s=1}^{2n+1}Z_s(\v_l)Z_l(\v_s)+4\left\langle J(\v),\nabla\left(\frac{\vt}{\eps}\right)\right\rangle+4n \frac{(\vt)^2}{\eps^2}.
    \end{equation}
Let us denote by $\nabla_S $ and $\Delta_S$ the gradient and the Laplace-Beltrami operator in $(S,g_\eps|_S)$ respectively.
A standard computation shows that
 \begin{equation}\label{lapbelt}
 \begin{split}
        &|\nabla_Sf|^2=|\nabla f|^2-g_\eps(\nabla f,\ve)^2\\
        &\Delta _S f=\sum_{i,j=1}^{2n+1}g^{i,j}Z_i(Z_j f)-H^\eps g_\eps(\nabla  f,\ve)+\frac{2\vet}{\eps}\langle\nabla  f,J(\ve)\rangle
        \end{split}
    \end{equation}
 for any $f\in C^2(S)$, where \begin{equation}\label{invmetric}
 g^{i,j}=\delta_{i,j}-\ve_i\ve_j.
 \end{equation}
 In particular, applying \eqref{lapbelt} to $\vet$ and exploiting \eqref{needed inproposition}, we infer that
    \begin{equation*}\label{precodige'}
        \begin{split}
            \Delta_S (\vet)&=g_\eps\left(\nabla_S H^\eps, Z_{2n+1}\right) -\vet\left(\ric(\ve,\ve)+|h^\eps|^2\right).
        \end{split}
    \end{equation*}

\section{Interior and global gradient estimates}\label{sec:ige}
Throughout this section, we fix $\eps\neq 0$ and the metric $g_\eps$, and we write as $\{Z_1,\ldots,Z_{2n},Z_{2n+1}\}$  the orthonormal frame $\{X_1,\ldots, X_n,Y_1,\ldots,Y_n,\eps T\}$. Moreover, we fix a bounded domain $\Om\subseteq \rr^{2n}$ and $H\in C^1(\Om)$.
We shall provide interior and global gradient estimates for $C^3$ solutions to \eqref{pmc'}.  Our approach follows the technique developed in \cite{MR0843597,MR0932680}. Given $u\in C^3(\Om)$, we denote by $\ve$ the Riemannian normal to the hypersurface $S=\graf(u)$, which can be globally extended to $\Om\times\rr$ through vertical translations by 
\begin{equation}\label{riemnormofgraph}
\ve(z,t)=-\sum_{i=1}^{2n}\frac{D_iu(z)+X_i(z)}{\sqrt{\eps^2+|Du(z)+X(z)|^2}}Z_j|_{(z,t)}+\frac{\eps}{\sqrt{\eps^2+|Du(z)+X(z)|^2}}Z_{2n+1}|_{(z,t)}
\end{equation}
for any $(z,t)\in\Om\times\rr$.
Moreover, given $f\in C^3(\Om)$, we can consider $f$ as a $C^3$ function on $S$ or on $\Om\times\rr$ by letting
    $f(z,u(z))=f(z)$ and $f(z,t)=f(z)$ respectively. In particular, it holds that $\nabla f=(Df,0)$. We begin with the following preliminary result.

\begin{proposition}
   Let $u\in C^3(\Om)$ be a classical solution to \eqref{pmc'} for $H\in L^\infty(\Om)\cap C^1(\Om)$. Then
    \begin{equation}\label{boundbelowlapl}
        \Delta _S u\geq -\|H\|_\infty(|\eps|+\max_{\bar\Om}|X|)-\frac{2}{|\eps|}\max_{\overline\Om}{|X|}.
    \end{equation} 
\end{proposition}
\begin{proof}
    Since $Z_{2n+1}u\equiv 0$, \eqref{riemnormofgraph} implies that
    \begin{equation*}
        \begin{split}
            \divv\left(\frac{Du+X}{\sqrt{\eps^2+|Du+X|^2}}\right)&=\frac{\Delta u}{\sqrt{\eps^2+|Du+X|^2}}+\left<Du+X,D\left((\eps^2+|Du+X|^2)^{-\frac{1}{2}}\right)\right>\\
            &=\frac{\Delta u}{\sqrt{\eps^2+|Du+X|^2}}-\frac{\sum_{i,j=1}^{2n}D_iD_ju(Du+X)_i(Du+X)_j}{(\eps^2+|Du+X|^2)^{\frac{3}{2}}}\\
            &=\frac{1}{\sqrt{\eps^2+|Du+X|^2}}\left(\sum_{i=1}^{2n+1}Z_iZ_i u-\sum_{i,j=1}^{2n+1}Z_iZ_ju\ve_i\ve_j\right)\\
            &=\frac{1}{\sqrt{\eps^2+|Du+X|^2}}\left(\sum_{i,j=1}^{2n+1}g^{i,j}Z_iZ_ju\right),
        \end{split}
    \end{equation*}
    and hence
    \begin{equation*}
        H\sqrt{\eps^2+|Du+X|^2}=\sum_{i,j=1}^{2n+1}g^{i,j}Z_iZ_ju.
    \end{equation*}
    Since $u$ solves \eqref{pmc'}, then our choice of $\ve$ in \eqref{riemnormofgraph} implies that $H^\eps = -H$. Hence, by \eqref{lapbelt},
    \begin{equation*}
    \begin{split}
        \Delta_S  u&=H\sqrt{\eps^2+|Du+X|^2}-\frac{H}{\sqrt{\eps^2+|Du+X|^2}}\langle D u,Du+X\rangle+\frac{2}{\sqrt{\eps^2+|Du+X|^2}}\langle D u,J(\ve)\rangle\\
        &=\frac{H}{\sqrt{\eps^2+|Du+X|^2}}\left(\eps^2+|Du+X|^2-\langle D u,Du+X\rangle\right)+\frac{2}{\sqrt{\eps^2+|Du+X|^2}}\langle D u,J(\ve)\rangle\\
        &=\frac{H}{\sqrt{\eps^2+|Du+X|^2}}\left(\eps^2+\langle X,Du+X\rangle\right)-\frac{2}{\sqrt{\eps^2+|Du+X|^2}}\langle X,J(\ve)\rangle\\
    \end{split}
    \end{equation*}
    Finally, \eqref{boundbelowlapl} follows at once from the previous computation.
\end{proof}
We are ready to prove \Cref{mainige'}.
\begin{proof}[Proof of \Cref{mainige'}]
Let $\tilde\Om$ and $\hat\Om$ be as in the statement, and let $r\in(0,d(\partial\hat\Om,\partial\tilde\Om))$ be fixed. In this way, $B(z_0,r)\Subset\tilde\Om$ for any $z_0\in\hat\Om$. Fix then $z_0\in\hat\Om$. We set $\gamma_1=\|u\|_{L^\infty(\tilde\Om)}$ and $\gamma_2=\|H\|_{C_1(\overline{\tilde\Om})}$.
    Let $\varphi\in C^\infty(\overline{B(z_0,r)})$ be the paraboloid centered at $z_0$ such that $\varphi(z_0)=u(z_0)-1$ and $\varphi(z)=\gamma_1$ for any $z\in\partial B(z_0,r)$, and let $\gamma_3=\|\varphi\|_{C^2(\overline{B(z_0,r)})}$.
    Notice that $\gamma_3=\gamma_3(r,\gamma_1)$. 
    We define 
\begin{equation*}
    \eta(t)=\left(e^{Kt}-1\right)e^{-(\gamma_1+\gamma_3)K}
\end{equation*}
for any $t\in\rr$, where $K>0$ is a constant to be chosen later. Notice that $\eta\in C^\infty(\rr)$ and  $0\leq\eta((u-\varphi)^+(z))\leq 1$ for any $z\in\overline{B(z_0,r)}$. Since the function $\Phi:\overline{B(z_0,r)}\longrightarrow \rr $ defined by
\begin{equation*}
    \Phi(z)= \frac{\varepsilon\cdot\eta((u-\varphi)^+(z))}{\vet(z)}
\end{equation*}
is continuous, we can denote by $M$ its maximum over $\overline{B(z_0,r)}$. Moreover, by the choice of $\varphi$, it holds that $\Phi\equiv 0$ on $\partial B(z_0,r)$ and $\Phi(z_0)>0$. Therefore the maximum $M$ is achieved at some point $\tilde z\in B(z_0,r)$. Since $M\geq\Phi(z_0)>0$, then $(u-\varphi)^+=u-\varphi$ locally near $\tilde z$ and $\Phi$ is of class $C^2$ in a neighborhood of $\tilde z$. In particular, 
\begin{equation*}
    \Psi(z):=\eta((u-\varphi)^+(z))-M\frac{\vet(z)}{\eps}\leq 0
\end{equation*}
for any $z\in\overline {B(z_0,r)}$, and 
\begin{equation}\label{fermatige}
    \Psi(\tilde z)=0,\qquad\nabla _S\Psi(\tilde z)=0\qquad\text{and}\qquad\Delta _S\Psi(\tilde z)\leq 0.
\end{equation}
We claim that there exists $M_0=M_0\left(\eps,\|X\|_{L^\infty(\Om)}\right)$ such that, if $f\in C^2(\Om)$ is any solution to \eqref{pmc'} satisfying 
\begin{equation}\label{supboundige}
    \|f\|_{L^\infty(B(z_0,r))}\leq\gamma_1,
\end{equation}
and $\varphi$, $M$ and $\tilde z$ are as above, then $M\geq M_0$ implies that
    \begin{equation}\label{fundstimaige}
        |\nabla_S (f-\varphi)|^2(\tilde z)>\frac{\eps^2}{2}.
    \end{equation}
    Indeed, let $(f_k)_k$ be a sequence such that
    \begin{equation}\label{contradictionige}
        M_k\geq k\qquad\text{and}\qquad |\nabla_S (f_k-\varphi_k)|^2(\tilde z)\leq \frac{\eps^2}{2},
    \end{equation}
where $M_k$ and $\tilde z_k$ are defined as above. Notice that
\[
\sqrt{\eps^2+|Df_k+X|^2(\tilde z_k)}\geq  \Phi(\tilde z_k)\geq k
\]
and, being $X$ bounded over $\overline\Om$, $|Df_k(\tilde z_k)|$ diverges to $+\infty$ as $k\to +\infty$. Moreover, by \eqref{lapbelt},
    \begin{equation*}
        \begin{split}
            |\nabla_S (f_k-\varphi_k)|^2&=|\nabla (f_k-\varphi_k)|^2-\langle\nabla  (f_k-\varphi_k),\ve\rangle^2\\
            &=|Df_k-D\varphi_k|^2-\frac{\langle Df_k-D\varphi_k,Df_k+X\rangle^2}{\eps^2+|Df_k+X|^2}\\
            &\geq\frac{\eps^2|Df_k-D\varphi_k|^2}{\eps^2+|Df_k+X|^2}.
        \end{split}
    \end{equation*}
    Hence, since $(D\varphi_k(\tilde z_k))_k$ bounded by \eqref{supboundige}, we get
    \begin{equation*}
        \liminf_{k\to+\infty}|\nabla_S (f_k-\varphi_k)(\tilde z_k)|^2\geq\lim_{k\to+\infty}\frac{\eps^2|Df_k-D\varphi_k|^2(\tilde z_k)}{\eps^2+|Df_k+X|^2(\tilde z_k)}=\eps^2,
    \end{equation*}
    which  contradicts \eqref{contradictionige}. We claim that $M\leq M_0$ for suitable choices of $K$. Indeed, for a fixed $K>0$, assume that $M>M_0$. Hence \eqref{fundstimaige} holds. Notice that $(u-\varphi)^+=u-\varphi$ locally around $\tilde z$.
Notice that $H^\eps(z,t)=-H(z)$ for any $z\in \Om$ since $u$ solves \eqref{pmc'}, and $H$ is extended vertically on $\Om\times \rr$. Exploiting \eqref{precodige}, \eqref{ricciexpress}, \eqref{boundbelowlapl}, \eqref{fermatige} and \eqref{fundstimaige}, and recalling that $u$ solves \eqref{pmc'}, we infer that, at $\tilde z$,
\begin{equation}\label{toxo}
    \begin{split}
        \Delta _S\Psi&=\Delta _S(\eta(u-\varphi))-\frac{M}{\eps}\Delta _S\vet\\
        &=\eta''|\nabla _S(u-\varphi)|^2+\eta'\left(\Delta _S u-\Delta _S\varphi\right)+\frac{M}{\eps}\big(g_\eps(\nabla_SH,\eps T)+\vet\left(\ric_\eps(\ve,\ve)+|h^\eps|^2\right)\big)\\
        &=\eta''|\nabla _S(u-\varphi)|^2+\eta'\left(\Delta _S u-\Delta _S\varphi\right)+\frac{M\vet}{\eps}\big(-g_\eps(\nabla H,\ve)+\ric_\eps(\ve,\ve)+|h^\eps|^2\big)\\
        &=\eta''|\nabla _S(u-\varphi)|^2+\eta'\left(\Delta _S u-\Delta _S\varphi\right)+\eta \big(-g_\eps(\nabla H,\ve)+\ric_\eps(\ve,\ve)+|h^\eps|^2\big)     \\
          &\geq\frac{\eps^2}{2}\eta''-C\eta'-C\eta
    \end{split}
\end{equation}
for some $C=C(n,\eps,r,\gamma_2,\gamma_3)>0$. Hence, since $\Delta _S\Psi(\tilde z)\leq 0$ and up to choosing a different constant $C=C(n,\eps,r,\gamma_2,\gamma_3)>0$, we conclude that
\begin{equation}\label{ODEige}
    \eta''((u-\varphi)(\tilde z))-C\eta'((u-\varphi)(\tilde z))-C^2\eta((u-\varphi)(\tilde z))\leq 0.
\end{equation}
The choice $K=2C$ is in contradiction with \eqref{ODEige}, so that, for this choice of $K$, we must have $M\leq M_0$. Hence 
\[
\eta((u-\varphi)^+(z))-M_0\frac{\vet(z)}{\eps}\leq 0.
\]
Since $(u-\varphi)^+(z_0)=1$, we conclude that
\begin{equation*}
    |Du|(z_0)\leq \sqrt{\eps^2+|Du+X|(z_0)^2}+|X|(z_0)\leq\frac{M_0}{\eta(1)}+|X|(z_0),
\end{equation*}
whence the thesis follows. 
\end{proof}
 An approach as in \Cref{mainige'} allows to reduce global gradient estimates for solutions to \eqref{pmc'} to boundary gradient estimates. For future convenience, we state the result for a slightly more general class of equations. 
\begin{theorem}[From boundary to global gradient estimates]\label{mainigeglobal}
    Let $H\in C^1(\overline\Om)$ and $\sigma\in [0,1]$, and let $u\in C^3(\Om)\cap C^1(\overline\Om)$ be a solution to 
    \begin{equation}\label{sigmaeq}
         \divv\left(\frac{Du+\sigma X}{\sqrt{\varepsilon^2+|Du+\sigma X|^2}}\right)=\sigma H(z)
    \end{equation}
    on $\Om$. Then there exists $C=C\left(n, \eps,\|u\|_{L^\infty(\Om)},\|H\|_{C^1(\overline\Om)},\|X\|_{L^\infty(\Om)}\right)>0,$ thus independent of $\sigma\in[0,1]$, such that
    \begin{equation*}\label{ige2024}
    \|Du\|_{L^\infty(\Om)}\leq C{\Big( \|Du\|_{L^\infty(\ptl\Om)}+1\Big)
    }.
    \end{equation*}
\end{theorem}
\begin{proof}
    Let us set $\gamma_1=\|u\|_{L^\infty(\Om)}$. Let $\eta(t)=e^{K(t-\gamma_1)}$, where $K>0$ is a fixed constant to be chosen later. Notice that 
    \begin{equation}\label{boundforetadue}
        e^{-2K\gamma_1}\leq \eta(u(z))\leq 1
    \end{equation}
    for any $z\in\overline\Om$.
     Let $\Phi:\bar\Om\to \rr$ be defined by
    \[
    \Phi(z)=\frac{\eps(\eta(u(z))}{\vet(z)}.
    \]
    Since $u\in C^1(\overline\Om)$, then $\Phi\in C(\overline\Om)$, so that $\Phi$ achieves its maximum $M$ at some point $\tilde z\in\overline\Om$. Suppose first $\tilde z\in\Om$. Arguing as in the proof of \Cref{mainige'}, and thanks to \eqref{boundforetadue}, there exists $M_0=M_0\left(\eps,\|X\|_{L^\infty(\Om)}\right)$ such that, if $f\in C^2(\Om)$ is any solution to \eqref{pmc'} satisfying \eqref{supboundige},
and $M$ and $\tilde z$ are as above, then $M\geq M_0$ implies that
    \begin{equation}\label{fundstimaige2}
        |\nabla_S f|^2(\tilde z)>\frac{\eps^2}{2}.
    \end{equation}
    Repeating the computations of \eqref{toxo},  and exploiting \eqref{fundstimaige2}, we get that, for a particular choice of $K=K(n,\eps,\|X\|_{L^\infty(\Om)},\|H\|_{C^1(\bar\Om)})>0$, it holds $M\leq M_0$. Hence, we get that
    \begin{equation*}
        \eta(u(z))\sqrt{\eps^2+|Du+X|(z)^2}\leq\eta(u(\tilde z))\sqrt{\eps^2+|Du+X|(\tilde z)^2}\leq M_0
    \end{equation*} 
   for any $z\in\overline\Om$, so that, by \eqref{boundforetadue},
\begin{equation*}
    |Du|(z)\leq\sqrt{\eps^2+|Du+X|(z)^2}+|X|(z)\leq M_0e^{2K\gamma_1}+\|X\|_{L^\infty(\Om)}
\end{equation*}
for any $z\in\overline\Om$. 
 If instead $\tilde z\in\partial\Om$,
    then
    \begin{equation*}
    \begin{split}
          |Du|(z)&\leq e^{2K\gamma_1}\eta(u(z))\sqrt{\eps^2+|Du+X|(z)^2}+|X|(z)\\
          &\leq e^{2K\gamma_1}\eta(u(\tilde z))\sqrt{\eps^2+|Du+X|(\tilde z)^2}+\|X\|_{L^\infty(\Om)}\\
          &\leq e^{2K\gamma_1}\left(|\eps|+\|Du\|_{L^\infty(\partial\Om)}+\|X\|_{L^\infty(\Om)}\right)+\|X\|_{L^\infty(\Om)}
    \end{split}
    \end{equation*}
  for any $z\in\overline\Om$, whence the thesis follows.
\end{proof}

As a corollary of \Cref{mainigeglobal}, we provide global gradient estimates for solutions to the Dirichlet problem associated with \eqref{pmc'} when \eqref{Serrinstrict} holds.
\begin{theorem}[Global gradient estimates]\label{mainigeglobaldue}
    Let $\Om$ be a bounded domain with boundary of class $C^2$.
    Let $H\in C^1(\overline\Om)$ be such that \eqref{Serrinstrict} holds.
    Let $\varphi\in C^2(\overline\Om)$, $\eps\neq 0$ and $\sigma\in [0,1]$.   Assume that there exist a constant $\gamma_1=\gamma_1(n,\eps,\Om,\varphi,X,H)>0$, independent of $\sigma\in[0,1]$, such that any solution $u\in C^2(\overline\Om)$ to 
    \begin{equation}\label{sigmaeqpmc2}
    \begin{cases} \divv\left(\frac{Du+\sigma X}{\sqrt{\varepsilon^2+|Du+\sigma X|^2}}\right)=\sigma H(z)&\text{ in  }\Om\\
			u=\sigma\varphi &\text{ in } \partial\Om
		\end{cases}
    \end{equation}
satisfies $\|u\|_{L^\infty(\Om)}\leq \gamma_1$. Then there exists $C=C\left(n, \eps,\|\varphi\|_{C^2(\overline\Om)},\gamma_1,\|H\|_{C^1(\overline\Om)},\|X\|_{L^\infty(\Om)}\right)>0,$ thus independent of $\sigma\in[0,1]$, such that any $u\in C^3(\Om)\cap C^1(\overline\Om)$ solution to \eqref{sigmaeqpmc2} satisfies
    \begin{equation*}\label{ige20243}
    \|Du\|_{L^\infty(\Om)}\leq C.
    \end{equation*}
\end{theorem}
\begin{proof}
    In view of \Cref{mainigeglobal} we are left to provide boundary gradient estimates. Thanks to \eqref{Serrinstrict}, and following \cite[Section 6]{pmc1}, the latter follow \emph{verbatim} as in \cite[Proposition 4.8]{pmc1}.
\end{proof}

\section{Existence and regularity of $t$-graphs}\label{sec:mainsection}
\subsection{Existence of minimizers: the non-extremal case}\label{subsec:existence}

Throughout this subsection we fix a bounded domain $\Om\subseteq\rr^{2n}$ with Lipschitz boundary, and we consider $H\in L^\infty(\Om)$ such that \eqref{Giusti} holds
and
 \begin{equation}\tag{***}\label{subextremal}
     \Big|\int_\Om Hdz\Big|<P(\Om).
 \end{equation}
We stress that, even without imposing boundary conditions, \eqref{Giusti} is a necessary condition to the existence of a solution $u\in C^2(\Om)$ to \eqref{pmc'} in $\Om$. More precisely, arguing as in the Euclidean setting, the following result holds.
\begin{theorem}\label{necessaryfirsthp}
    Let $\Om\subseteq\rr^{2n} $ be a bounded domain with Lipschitz boundary, let $H\in\lip(\Om)$ and assume that there exists $u\in C^2(\Om)$ which solves \eqref{pmc'} in $\Om$. Then \eqref{Giusti} holds.
\end{theorem}
For any $\varphi\in L^1(\partial\Om)$, we define the functional $\mathcal I_\varepsilon:BV(\Om)\to \rr$ by
 \begin{equation} \label{ieps}
\mathcal{I}_\eps(v)=P_\eps(E_v,\Om\times\rr)+\int_\Om Hu\,dz+\int_{\partial\Om}|v-\varphi|\,d\mathcal{H}^{2n-1},
 \end{equation}
where $E_v$ is defined as in \eqref{def:subgraph}. 
The following result follows as \cite[Lemma 1.1]{MR487722} (cf. also \cite[Lemma 3.2]{MR3767675}).
\begin{lemma}\label{lem:nonop}
 Assume that \eqref{Giusti} and \eqref{subextremal} hold. 
 Then there exists $\delta>0$ such that
 \begin{equation}\label{eq:nonop}
 \Big|\int_{\tilde \Om} H\,dz\Big|\leq (1-\delta)P(\tilde \Om)
 \end{equation}
 for every measurable set $\tilde \Om\subseteq\Om$.
\end{lemma}

The proof of the following proposition reproduces the argument of \cite[Theorem 1.1]{MR336532}.

\begin{proposition}\label{existenceeps}
Let $H\in L^\infty(\Om)$ and assume that \eqref{Giusti} and \eqref{subextremal} hold. Then $\mathcal I_\varepsilon$
 has a minimum in $BV(\Om)$  for every $\varphi\in L^1(\partial\Om)$.  
\end{proposition}

\begin{proof}
 Let $B\subseteq\rr^{2n}$ be a ball containing $\Om$ such that the Euclidean distance between $\partial \Om$ and $\partial B$ is positive, and extend $H$ to $B$ by letting $H\equiv 0$ outside $\Om$. Fix a function $\phi\in W_0^{1,1}(B)$ with trace $\varphi$ on $\partial\Om$. Then minimizing $\mathcal{I}_\eps$ is equivalent to minimize the functional
 \begin{equation}\label{def:functionalJ}
 \mathcal{J}_\eps(v)=P_\eps(E_v,\Om\times\rr)+Var(v,B\setminus\Om)+ \int_B H v\, dz
 \end{equation}
 in $K=\{v\in BV(B) : v=\phi \text{ in }B\setminus\Om \}$. Indeed, given $v_0$ and $v$ in $K$ with $v_0$ minimum of $\mathcal{J}_\eps$ and writing $u_0=v_0|_{\Om}$ and $u=v|_{\Om}$, it follows from \cite[Remark 2.13]{MR0775682} that
 \begin{equation*}
 \mathcal{I}_\eps(u)- \mathcal{I}_\eps(u_0)=\mathcal{J}_\eps(v)-\mathcal{J}_\eps(v_0)\geq 0.
 \end{equation*}
  For a given $v\in K$, we define $v^+,v^-\in BV(B)$ by
  
    \begin{equation*}
v^+(z)=
\displaystyle{\begin{cases}
\max\{0,v(z)\}&\text{ if }z\in\Om,\\
0&\text{otherwise}
\end{cases}
}\qquad\text{and}\qquad v^-(z)=
\displaystyle{\begin{cases}
\max\{0,-v(z)\}&\text{ if }z\in\Om,\\
0&\text{otherwise}
\end{cases}
}.
\end{equation*}
Notice that
\begin{equation}\label{etinom}
   E_t:= \{z\in B\,:\,v^+(z)>t\}=\{z\in\Om\,:\,v(z)>t\}\subseteq\Om 
\end{equation}
and
\begin{equation}\label{ftinom}
   F_t:= \{z\in B\,:\,v^-(z)>t\}=\{z\in\Om\,:\,v(z)<-t\}\subseteq\Om
\end{equation}
for any $t>0$. 
Hence, thanks to \Cref{lem:nonop}, \eqref{etinom}, \eqref{ftinom}, the \emph{layer-cake formula} (cf. \cite[Remark 13.6]{MR2976521}) and the \emph{Coarea formula} (cf. \cite[Theorem 13.1]{MR2976521}), it follows that 
 \begin{equation*}
 \begin{split}
\int_B H(z)v(z)\,dz&=\int_\Om H(z)v^+(z)\,dz-\int_\Om H(z)v^-(z)\,dz\\
&=\int_0^{+\infty}\,dt\int_{\Om}\chi_{E_t}H(z)\,dz-\int_{0}^{+\infty}\,dt\int_{\Om}\chi_{F_t}H(z)\,dz\\
&\geq-(1-\delta)\int_0^{+\infty}P(E_t)\,dt-(1-\delta)\int_{0}^{+\infty}P(F_t)\,dt\\
&=-(1-\delta)\int_0^{+\infty}P(E_t,B)\,dt-(1-\delta)\int_{0}^{+\infty}P(F_t,B)\,dt\\
&\geq-(1-\delta)Var(v^+,B)-(1-\delta)Var(v^-,B).
 \end{split}
 \end{equation*}
 Therefore, owing again to \cite[Remark 2.13]{MR0775682},
 \begin{equation}\label{in:intH}
     \begin{split}
     \int_B& H(z)v(z)dz\\
     &\geq -(1-\delta)(Var(v^+,\Om)+Var(v^-,\Om))-(1-\delta)\left(\int_{\partial\Om}|v^+|\,d\mathcal{H}^{2n-1}+\int_{\partial\Om}|v^-|\,d\mathcal{H}^{2n-1}\right)\\
    &= -(1-\delta)Var(v,\Om)-(1-\delta)\int_{\partial\Om}|v|\,d\mathcal{H}^{2n-1}\\
    &\geq-(1-\delta)Var(v,\Om)-(1-\delta)\int_{\partial\Om}|v-\varphi|\,d\mathcal{H}^{2n-1}-(1-\delta)\int_{\partial\Om}|\varphi|\,d\mathcal{H}^{2n-1}\\
    &=-(1-\delta)Var(v,\overline\Om)-(1-\delta)\int_{\partial\Om}|\varphi|\,d\mathcal{H}^{2n-1}\\
    &\geq-(1-\delta)Var(v,B)-\int_{\partial\Om}|\varphi|\,d\mathcal{H}^{2n-1}\\
     \end{split}
 \end{equation}
 On the other hand, it holds that 
 \begin{equation}\label{in:pergraph}
     P_\eps(E_v;\Om\times\rr)\geq Var(v,\Om)-\int_\Om|X|\,dz.
 \end{equation}
 Indeed, let $(v_{k})_{k}\subseteq W^{1,1}(\Om)$ be such that $v_{k}\to v$ in $L^1(\Om)$. Then, for any fixed ${k}\in\mathbb N$, \eqref{periforw11} implies that
 \begin{equation*}
     \begin{split}
         P_\eps(E_{v_{k}};\Om\times\rr)&=\int_\Om\sqrt{\eps^2+|D v_{k}+X|^2}\,dz\geq\int_\Om|Dv_{k}+X|\,dz\geq\int_\Om|D v_{k}|\,dz-\int_\Om|X|\,dz
     \end{split}
 \end{equation*}
 Hence, by the lower semicontinuity of $Var(\cdot,\Om)$ with respect to the $L^1$-convergence, we conclude that
 \begin{equation*}
     \int_\Om|D v|\,dz-\int_\Om|X|\,dz\leq\liminf_{k\to+\infty}\left(\int_\Om|D v_{k}|\,dz-\int_\Om|X|\,dz,\right)\leq \liminf_{k\to+\infty}P_\eps(E_{v_{k}};\Om\times\rr),
 \end{equation*}
 from which \eqref{in:pergraph} follows by \Cref{equifun}.
Substituting \eqref{in:intH} and \eqref{in:pergraph} in \eqref{def:functionalJ}, we finally obtain
 \begin{equation}\label{eq:BVbound}
 \mathcal{J}_\eps(v)\geq \delta Var(v,B)-\int_\Om|X|\,dz-\int_{\partial\Om}|\varphi|\,d\mathcal{H}^{2n-1}
 \end{equation}
Let $\{v_k\}\subseteq K$ be a minimizing sequence for $\mathcal J_\eps$. By \eqref{eq:BVbound} and the Poincaré's inequality, $\{v_k\}$ is bounded in $BV(B)$, so that, up to a subsequence, there exists $v_0\in BV(B)$ such that
$v_k\to v_0$ in $L^1(B)$.  Noticing that $K$ is a closed with respect to the $L^1$ convergence on $B$, then $v_0\in K$. By the lower semicontinuity of the perimeter with respect to $L^1$ convergence, we get that $v_0$ is a minimizer of $\mathcal{J}_\eps$ and hence $\mathcal{I}_\eps$ has a minimizer in $BV(\Om)$.
\end{proof}

\subsection{Variational properties of minimizers}\label{densityestimates}
Throughout this subsection, we fix a bounded domain $\Om\subseteq\rr^{2n}$ and $H\in L^\infty_{loc}(\Om)$.
\begin{definition}\label{hmindeff}
 We say that a Caccioppoli set $E\subseteq\Om\times\rr$ is a (local) \emph{$H$-minimizer} in $\Om\times\rr$ if 
     \begin{equation}\label{epsminmaggio}
        P_\varepsilon (E,A)+\int_{E\cap A}Hdx\leq P_\eps(F,A)+\int_{F\cap A}Hdx
    \end{equation}
    for any open  set $A\Subset\Om\times\rr$ and any measurable set $F$ such that $E\Delta F\Subset A$, where $H(z,t)=H(z)$.
\end{definition}

In the following theorem we prove that subgraphs of minimizers of $\mathcal{I_\varepsilon}$ are $H$-minimizers (cf. \cite{MR0816345} for an exhaustive account of this and related results in the Euclidean setting). Notice that, if $u\in BV(\Om)$ is a minimizer for $\mathcal I_\varepsilon$, then
\begin{equation}\label{limminloc2esp}
\int_{\tilde\Om}\sqrt{\eps^2+|Du+X|^2}+\int_{\tilde\Om}Hu\,dz\leq \int_{\tilde\Om}\sqrt{\eps^2+|Dv+X|^2}+\int_{\tilde\Om}Hv\,dz
    \end{equation}
for any open set $\tilde\Om\Subset\Om$ and any $v\in BV_{loc}(\Om)$ such that $\{u\neq v\}\Subset\tilde\Om$. When $u\in BV_{loc}(\Om)$ satisfies \eqref{limminloc2esp}, we refer to it as \emph{(local) $H$-minimizer}. The ambiguity with \Cref{hmindeff} is motivated by the following result.

\begin{theorem}
\label{minimplieslambdamin}
    Let $u\in BV_{loc}(\Om)$. Then $u$ satisfies \eqref{limminloc2esp} if, and only if, $E_u$ is an $H$-minimizer in $\Om\times\rr$.
\end{theorem}

\begin{proof}
Assume first that $u\in BV_{loc}(\Om)$ satisfies \eqref{limminloc2esp}. Let $A\Subset\Om\times\rr$ be open, and let $\tilde\Om\Subset\Om$ be an open set such that $A\Subset\tilde\Om\times\rr$. let now $F$ be a measurable set such that $F\Delta E_u\Subset A$. We can assume without loss of generality that $F$ has finite perimeter in $A$. Since $E_u$ is a subgraph, we infer that
\begin{equation*}\label{pseudograph}
\lim_{t\to+\infty}\chi_F(z,t)=0\qquad\text{and}\qquad\lim_{t\to-\infty}\chi_F(z,t)=1
    \end{equation*}
    for a.e. $z\in\Om$.
Inspired by \cite[Lemma 14.7]{MR0775682}, we set $w\in BV_{loc}(\Om)$ the function
    \begin{equation*}\label{def:w}
        w(z)=\lim_{k\to\infty}w_k(z)=\lim_{k\to\infty}\left(\int_{-k}^k\chi_F(z,t)\,dt-k\right).
    \end{equation*}
We claim that
\begin{equation}\label{pseudograph2}
        P_\eps(E_w;\tilde\Om\times \rr)\leq P_\eps(F,\tilde\Om\times\rr)
    \end{equation}
    and
    \begin{equation}\label{samearea}
        \int_{\Om'\times\rr}\left(\chi_{E_w}-\chi_F\right)\,dz\,dt=0
    \end{equation}
for any open set $\Om'\Subset \tilde\Om$. Indeed, assume first that $\exists L>0$ such that 
\begin{equation}\label{hpaux}
        \Om\times(-\infty,-L)\subseteq F\subseteq\Om\times(-\infty,L).
    \end{equation}
     It is clear that $-L\leq w\leq L$ and $w(z)=-L+\int_{-L}^L\chi_F(z,t)\,dt$.  Let $\eta:\rr\to[0,1]$ be a smooth cut-off function between $[-L,L]$ and $[-L-1,L+1]$. From a direct computation, we get that for a.e. $z\in\Om$, it holds
    \begin{equation}\label{aux2}
        \int_\rr\chi_F(z,t)\eta(t)\,dt=w(z)+\alpha\qquad\text{and}\qquad\int_\rr\chi_F(z,t)\eta'(t)\,dt=1,
    \end{equation}
    where $\alpha=L+\int_{-L-1}^L\eta(t)\,dt$. Let $g=(\tilde g,g_{2n+1})\in C_c^1(\Om,\rr^{2n+1})$ be such that $|g|\leq 1$. Let us set 
    \begin{equation*}
        W|_{(z,t)}=\sum_{j=1}^{2n}\eta(t)g_j(z)Z_j|_{(z,t)}.
    \end{equation*}  Then $W\in C_c^1(\Om\times\rr,T\hh^n)$ and $|W|_\varepsilon\leq 1$. Therefore \eqref{aux2} implies that
    \begin{equation*}
        \begin{split}
            P_\varepsilon (F,\Om\times\rr)\geq&\int_{\Om\times\rr}\chi_F(z,t)\divv W(z,t)\,dt\,dz\\
            =&\int_{\Om}\int_\rr\chi_F(z,t)\left(\eta(t)\divv\tilde g(z)-\eta'(t)\langle\tilde g(z),X(z)\rangle+\eta'(t)\varepsilon g_{2n+1}(z)\right)\,dt\,dz\\
            =&\int_\Om  \Big( \divv\tilde g(z) \int_\rr\chi_F(z,t)\eta(t)\,dt-\langle\tilde g(z),X(z)\rangle\int_\rr\chi_F(z,t)\eta'(t)\,dt\\
            &+\varepsilon g_{2n+1}\int_\rr\chi_F(z,t)\eta'(t)\,dt\Big)\,dz\\
            =&\int_\Om (w(z)+\alpha)\divv\tilde g(z)-\langle\tilde g(z),X(z)\rangle+\varepsilon g_{2n+1}(z)\,dz\\
            =&\int_\Om w(z)\divv\tilde g(z)-\langle\tilde g(z),X(z)\rangle+\varepsilon g_{2n+1}(z)\,dz,
        \end{split}
    \end{equation*}
    where we used the fact that $\text{supp}(\tilde g)\subseteq \Om$. Hence, assuming \eqref{hpaux}, \eqref{pseudograph2} holds by \Cref{equifun}. Moreover, 
    \begin{equation}\label{julianfine}
              \begin{split}
        \int_{\tilde\Om}H(u-w)dz=&\int_{\tilde\Om\cap\{u\geq w\}}H(u-w)dz-\int_{\tilde\Om\cap\{u<w\}}H(w-u)dz\\
        =&\int_{\tilde\Om\cap\{u\geq w\}}H(|E_u(z)\cap A(z)|-|E_w(z)\cap A(z)|)dz\\
        &-\int_{\tilde\Om\cap\{u<w\}}H(|E_w(z)\cap A(z)|-|E_u(z)\cap A(z)|)dz\\
        =&\int_{(\tilde\Om\cap\{u\geq w\})\times\rr}H(\chi_{E_u\cap A}-\chi_{E_w\cap A})dx-\int_{\tilde\Om\cap\{u<w\}}H(\chi_{E_w\cap A}-\chi_{E_u\cap A})dx\\
        =&\int_{\tilde\Om\times \rr}H(\chi_{E_u\cap A}-\chi_{E_w\cap A})dx,
   \end{split}
    \end{equation}
where $E_v(z)=\{t\in\rr : v(z)>t \}$ and $A(z)=\{t\in\rr : (z,t)\in A \}$.    
To drop \eqref{hpaux}, one can argue exactly as in the proof of \cite[Theorem 14.8]{MR0775682}.
Finally, \eqref{samearea} follows \emph{verbatim} as in the proof of \cite[Teorema 2.3]{MR174706}.
    In view of \eqref{pseudograph2}, \eqref{samearea} and \eqref{julianfine}, arguing as in \cite[Theorem 14.9]{MR0775682} and \cite{MR0500423} $E_u$ is an $H$-minimizer in $\Om\times\rr$. Being the converse implication fairly straightforward, the thesis follows.
\end{proof}

Clearly, since $H\in L^\infty_{loc}(\Om)$, $H$-minimizer sets are \emph{almost minimizers} in the sense of \cite{MR0420406} for suitable \emph{anisotropic energies}. More precisely, if $E$ is an $H$-minimizer in $\Om\times\rr$, $\tilde\Om\Subset\Om$ and $\|H\|_{L^\infty(\tilde\Om)}=H_0$, then
\begin{equation}\label{quasiminimizerdef}
   P_\varepsilon (E,A)\leq P_\varepsilon (F,A)+H_0|E\Delta F|
\end{equation}
for any open set $A\Subset\tilde\Om\times\rr$ and any measurable set $F\subseteq \Om\times\rr$ such that $E\Delta F\Subset A$. For almost minimizers in the previous sense, arguing as in the Euclidean setting (cf. \cite{MR2976521}) it is possible to derive \emph{uniform density estimates} for both volume and perimeter. In view of \eqref{ineq:perieq}, the following well-known estimates follow from for instance from \cite[Proposition 4.5]{MR4430590}. 



\begin{proposition}\label{prop:voldensity}
    For any domain  $\tilde\Om\Subset\Om$, there exist $c_0=c_0(\eps)>0$, $c_1=c_1(\eps)>0$, $c_2>0$ independent of $\eps\in(0,1]$ and $r_0>0$ such that for any $H$-minimizer $E$, any $p\in \overline{\ptl^* E}\cap(\tilde\Om\times\rr)$ and any $r<r_0$, it holds
    \begin{equation}\label{in:voldensity}
   c_0 r^{2n+1}\leq \min\{|E\cap B(p,r)|,|E\setminus B(p,r)|\}
   \end{equation}
   and
   \begin{equation}\label{ineq:perden}
c_1 r^{2n}\leq P_\varepsilon(E,B(p,r))\leq c_2 r^{2n},
    \end{equation}
    where $B(p,r)$ is the Euclidean ball centered at $p$ of radius $r$.
\end{proposition}

The local boundedness of minimizers is a consequence of \Cref{minimplieslambdamin} and \Cref{prop:voldensity}, and it follows as its Euclidean counterpart (cf. \cite[Theorem 14.10]{MR0775682}).

\begin{proposition}\label{linftyloc}
    Let  $u\in BV(\Om)$ be a minimizer of $\mathcal I_\varepsilon$. Then $u\in L^\infty_{loc}(\Om)$.
\end{proposition}
\subsection{Higher regularity of Lipschitz continuous t-graphs}
In the rest of this section, inspired by \cite{MR0775682}, we show classical regularity for minimizers of $\mathcal{I_\varepsilon}$. 
The main difficulty consists in obtaining Lipschitz regularity since, as soon as a minimizer is Lipschitz, standard regularity results for uniformly elliptic equations apply. More precisely, the following regularity property for Lipschitz weak solutions to \eqref{pmc'} holds.

\begin{theorem}\label{uisc2}
    Let $H\in\lip_{loc}(\Om)$ and $u\in \lip_{loc}(\Om)$ be a weak solution on $\Om$ to \eqref{pmc'}.
     Then $u\in C^{2,\alpha}_{loc}(\Om)$ for any $\alpha\in (0,1)$ and is a classical solution to \eqref{pmc'}.
    Moreover, if  
    $H\in C^{k,\gamma}_{loc}(\Om)$ for some $k\geq 1$ and $\gamma\in(0,1)$, then $u\in C^{k+2,\gamma}_{loc}(\Om)$. Finally, if $H\in C^\infty(\Om)$, then $u\in C^\infty(\Om)$.
\end{theorem}

\begin{proof}
    Let $\alpha\in (0,1)$. Let $\tilde\Om\Subset\Om$ be a bounded domain. It suffices to show that $u\in C^{2,\alpha}  	(\tilde\Om)$. For $r>0,$ set 
    $$
    \tilde\Om_r=\left\{z\in \tilde\Om\,:\,\min_{w\in\partial\Om}|z-w|>r\right\}.
    $$
	Given  $\psi\in C^\infty_c(\tilde\Om)$, we can assume that $\psi\in C^\infty_c(\tilde\Om_r)$ for $r>0$ small enough.   	Let $v\in B(0,r)$.
    Using $\psi$ and $\psi(\cdot-v)$ as test functions the weak formulation of \eqref{pmc'}, we get
    \begin{equation}\label{eq:weak}
    \int_{\tilde\Om}\escpr{A(z+v, Du(z+v))-A(z, Du(z)), D\psi(z)}\,dz=\int_{\tilde\Om}(H(z)-H(z+v))\psi(z)\,dz,
    \end{equation}
    where
    \begin{equation*}
        A(z,\xi)=\frac{\xi+X(z)}{\sqrt{\varepsilon^2+|\xi+X(z)|^2}}.
    \end{equation*}
     Fixed $z_0\in\tilde\Om_r$, the fundamental theorem of calculus implies that
\begin{equation}\label{eq:weak2}
    A(z_0+v, Du(z_0+v))-A(z_0, Du(z_0))=\int_0^1 \frac{d(A(\alpha_{z_0}(s))}{ds}dt,
\end{equation}
where $\alpha_{z_0}:[0,1]\to \rr^{2n}\times \rr^{2n}$ is given by
    \begin{equation*}
        \alpha_{z_0}(t)=(z_0+tv,D u(z_0)+t(D u(z_0+v)-D u(z_0))).
    \end{equation*}
    Writing $u_v(z)=\frac{u(z+v)-u(z)}{|v|}$, it follows from a direct computation that
    \begin{equation}\label{eq:weak3}
    \begin{split}
         \frac{d(A(\alpha_z(t))}{dt}&=A_z(\alpha_z(t))\cdot v+|v|A_\xi(\alpha_z(t))\cdot D u_v(z),
    \end{split}
    \end{equation}
    where $A_\xi$ is the matrix with entries $(A_\xi)_{ij}=\frac{\partial A_j}{\partial z_i}$ given by
    \begin{equation}\label{shapeexplicit}
        (A_\xi)_{ij}(z,\xi)=\frac{\delta_{ij}(\eps^2+|\xi+ X(z)|^2)-(\xi_i+X(z)_i)(\xi_j+X(z)_j)}{(\eps^2+|\xi+ X(z)|^2)^{\frac{3}{2}}}.
    \end{equation} 
We set
\begin{equation*}
\begin{split}
    \tilde A(z)&= \int_0^1A_\xi(\alpha_z(t))\,dt,\\
    \tilde B_v (z)&=\left(\int_0^1A_z(\alpha_z(t))\,dt\right)\cdot\frac{v}{|v|}\\
    \tilde H_v(z)&=\frac{H(z)-H(z+v)}{|v|}.
    \end{split}
\end{equation*}
Inserting \eqref{eq:weak2} and \eqref{eq:weak3} in \eqref{eq:weak} and dividing by $|v|$, we get
\begin{equation*}
    \int_{\tilde\Om}\langle \tilde A(z)\cdot D u_v(z),D\psi(z)\rangle\,dz=\int_{\tilde\Om}\tilde H_v(z)\psi(z)\,dz-\int_{\tilde\Om}\left\langle\tilde B_v(z),D\psi(z)\right\rangle\,dz.
\end{equation*} 
In other words, $u_v$ is a weak solution on $\tilde\Om$ to the linear equation
\begin{equation*}
    \divv\left(\tilde A\cdot D u_v\right)-\tilde H_v-\divv\tilde B_v.
\end{equation*}
Since $H\in \lip(\tilde\Om)$, the coefficients of $\tilde A$, $\tilde B_v$ and $\tilde H_v$ are uniformly bounded with respect to $v$. Moreover, since $u\in\lip(\tilde\Om)$, \eqref{shapeexplicit} implies that $\tilde A$ is uniformly elliptic on $\tilde\Om$. Finally, since $u\in \lip (\tilde\Om)$, then $u_v$ is bounded in $L^\infty(\tilde\Om)$ uniformly for $v\in B_r$ with $r>0$ small enough. Therefore, using the celebrated De Giorgi-Nash-Moser method (cf. \cite[Chapter 4, Theorem 2.3]{MR1616087} and cf. the proof of the Corollary right after \cite[Chapter 4, Theorem 2.2]{MR1616087}), we conclude that, up to a smaller $\tilde\Om$, $u_v$ is bounded in $C^{0,\alpha}(\tilde\Om)$ uniformly for $v\in B_r$, so that $u\in C^{1,\alpha}(\tilde\Om)$ and $u_v\in C^{1,\alpha}(\tilde\Om)$. In particular, $\tilde A$ and $\tilde B_v$ are of class $C^{0,\alpha}$. Therefore, in view of \cite[Theorem 8,32]{GT} and up to a smaller $\tilde\Om$, $u_v$ is bounded in $C^{1,\alpha}(\tilde\Om)$ uniformly for $v\in B_r$, so that $u\in C^{2,\alpha}(\tilde\Om)$. In particular, $u$ is a classical solution to \eqref{pmc'}. 
Let $k\geq 1$ and assume $H\in C^{k,\gamma}_{loc}(\Om)$. Since $u$ is a weak solution to \eqref{pmc'}, by means of \cite[Chapter 12, Theorem 1.1]{MR1616087} we infer that, for any $j=1,\ldots,2n$, $g:=(Du)_j$ is a weak solution to the linear equation
\begin{equation*}
    \divv\left(\frac{\partial A}{\partial \xi}(z,Du)\cdot Dg\right)=D_jH-\divv\left(\frac{\partial A}{\partial z_j}(z,Du)\right).
\end{equation*}
The thesis then follows exploiting the classical Schauder's theory (cf. \cite[Theorem 5.20]{MR3099262}).
\end{proof}
\subsection{The Dirichlet problem for \eqref{pmc'}}

As a corollary of the global gradient estimates, we extend the existence result obtained in \cite{pmc1} for the sub-Finsler constant mean curvature equation to the existence of solutions for the sub-Riemannian prescribed, but not necessarily constant, mean curvature equation. 

\begin{proof}[Proof of \Cref{existence*}]
Assume first that $H\in C^{1,\alpha}(\overline\Om)$.
Arguing exactly as in \cite{pmc1}, it suffices to provide \emph{a priori} estimates in $C^1(\overline\Om)$ for solutions $u\in C^{2,\alpha}(\overline\Om)$ to \eqref{sigmaeq} with boundary datum $\varphi$ which are independent of $u$ and $\sigma\in[0,1]$. First, by means of \Cref{lem:nonop}, uniform estimates in $L^\infty(\Om)$ follow as in \cite[Proposition 4.4]{pmc1}. Moreover, arguing \emph{verbatim} as in \Cref{uisc2}, any solution $u\in C^{2,\alpha}(\overline\Om)$ to \eqref{sigmaeq} belongs to $C^3(\Om)\cap C^1(\overline\Om)$. Hence, \Cref{mainigeglobaldue} provides global gradient estimates uniform with respect to $\sigma\in[0,1]$. 
 Assume now that $H\in\lip(\Om)$, and denote by $H_0$ its Lipschitz constant. By McShane's extension theorem, we can suppose that $H\in\lip(\rr^{2n})$ with the same Lipschitz constant $H_0$. By a standard mollification argument, there exists a sequence $(H_j)_j\subseteq C^\infty(\overline\Om)$ such that
\begin{equation}\label{approxhlip}
    H_j\to H \text{ uniformly on } \overline\Om\qquad\text{and}\qquad\|H_j\|_{C^1(\overline\Om)}\leq H_0+\|H\|_{L^\infty(\Om)}+1
\end{equation}
for any $j\in\mathbb N$. Since $H$ and $\Om$ satisfy \eqref{Giusti} and \eqref{subextremal}, we let $\delta$ be as in the statement of \Cref{lem:nonop}.  Let us denote by $h(\Om)$ the \emph{Cheeger constant} of $\Om$ (cf. \cite{MR3467379}), that is 
\begin{equation*}
    h(\Om)=\inf\left\{\frac{P(A)}{|A|}\,:\,A\subseteq\Om,\,|A|>0\right\}.
\end{equation*}
Being $\Om$ bounded and open, it is well known (cf. \cite[Proposition 3.5]{MR3467379}) that $h(\Om)>0$. By \eqref{approxhlip}, we can assume up to a subsequence that
\begin{equation*}
    \|H-H_j\|_{L^\infty(\Om)}\leq \frac{\delta h(\Om)}{2},
\end{equation*}
so that
 \begin{equation}\label{stima1perlip}
 \Big|\int_{A} H_jdz\Big|\leq  \Big|\int_{A} Hdz\Big|+|A|\|H-H_j\|_{L^\infty(\Om)}\leq\left(1- \frac{\delta}{2}\right)P(A)
 \end{equation}
for any $A\subseteq\Om$ such that $|A|\neq 0$. On the other hand, by \eqref{Serrinstrict} and \eqref{approxhlip}, there exists $C=C(\Om,\|H\|_{L^\infty(\Om)})>0$ such that, up to a subsequence,
\begin{equation}
\label{stima1perlip2}
    |H_j(z_0)|\leq H_{\partial\Om}(z_0)-C
\end{equation}
for any $z_0\in\partial\Om$ and any $j\in\mathbb N$. Combining \eqref{stima1perlip} and \eqref{stima1perlip2}, from the previous step we get a solution $u_j\in C^{2,\alpha}(\overline\Om)$ to \eqref{pmc'} with boundary datum $\varphi$ and source $H_j$. Again by \eqref{approxhlip}, \eqref{stima1perlip} and \eqref{stima1perlip2}, and following \cite[Proposition 4.4]{pmc1}, \cite[Proposition 4.8]{pmc1} and \Cref{mainigeglobaldue}, $(u_j)_j$ is uniformly bounded in $C^1(\overline\Om)$, so that, by Ascoli-Arzelà Theorem, there exists $u\in \lip(\Om)$ such that $u_j\to u$ uniformly on $\overline\Om$. First, notice that $u=\varphi$ on $\partial\Om$. Finally, a compactness argument as the forthcoming \Cref{hcompthm}, coupled with \Cref{minimplieslambdamin} and \Cref{uisc2}, implies that $u\in C^{2,\alpha}_{loc}(\Om)$ and that $u$ is a classical solution to \eqref{pmc'}, whence the thesis follows.
\end{proof}

\subsection{Lipschitz regularity of $t$-graphs}
Throughout this subsection, we fix a bounded domain $\Om\subseteq\rr^{2n}$ with Lipschitz boundary and $H\in C^{1,\ga}_{loc}(\Om)$ for some $\ga\in(0,1)$. We are left to show that $BV$-minimizers of $\mathcal{I}_\varepsilon$ are locally Lipschitz continuous.

\begin{proposition}\label{reglambdamin}
     Let  $H\in C^{1,\gamma}_{loc}(\Om)$  and let $u\in BV(\Om)$ be such that $E_u$ is an $H$-minimizer for $P_\eps$ on $\Om\times\rr$.
     Then $\partial^* E_u$ is a $C^{3,\gamma}$ manifold
     and $\mathcal H^{s}(\partial E\setminus\partial^*E)=0$ for any $s>2n-7$.
\end{proposition}
\begin{proof}
Let $u$ be as in the statement. Let $\tilde\Om\Subset\Om$ be an open set. Fix $\alpha\in(0,1)$ and let $F\Subset \tilde\Om\times\rr$ be a finite perimeter set  such that $E\Delta F\Subset B(p,r)\Subset\tilde\Om\times\rr$ for some $p\in\tilde\Om\times\rr$ and $r\in(0,1)$. Taking $H_0=\|H\|_{L^\infty(\tilde\Om)}$, it follows that
\begin{equation*}
    \begin{split}
P_\varepsilon(E,B(p,r))&\leq P_\varepsilon(F,B(p,r))+H_0|E\Delta F|\leq 
P_\varepsilon(F,B(p,r))+H_0w_nr^{2n+\alpha}.
    \end{split}
\end{equation*}
Moreover, for any Caccioppoli set $F$ in $\Om\times\rr$ and any open set $A\subseteq\Om\times\rr$ it holds that
\begin{equation*}
    P_\varepsilon(F,A)=\int_{\partial^\star F\cap A}|C_\varepsilon(p)^T\nu(p)|\,d\mathcal H^{2n}(p)=\int_{\partial^\star F\cap A}\sqrt{\langle M_\varepsilon(p)\nu(p),\nu(p)\rangle}\,d\mathcal H^{2n}(p),
\end{equation*}
where $C_\eps$ is defined in \eqref{base}, $M_\varepsilon(p)=C_\varepsilon(p)^T\cdot C_\varepsilon(p)$ for any $p\in\Om\times\rr$ and $\nu$ is the measure theoretic Euclidean unit normal to $F$. It is easy to check that $ M_\varepsilon $ is uniformly positive definite and $\alpha$-H\"older continuous. Then \cite[Theorem 1.1]{MR4430590} implies that $E_u$ is a manifold of class $C^{1,\frac{\alpha}{4}}$,
     and moreover $\mathcal H^{s}(\partial E\setminus\partial^*E)=0$ for any $s>2n-7$.  
     Fix $0<\alpha<1$ such that $\frac{\alpha}{4}\leq \gamma$. Since $S$ is $C^{1,\frac{\alpha}{4}}$, for any $p\in S$ we can consider an open neighborhood $U$ of $p$ in the tangent hyperplane of $S$ at $p$ where $S$ coincides with $\Phi_{\tilde u}(U):=\{q-\nu(p)\tilde u(q): q\in U\}$ for a suitable function $\tilde u$ of class $C^{1,\frac{\alpha}{4}}$. Given $v\in W^{1,1}(U)$, we let $A(v)$ be the area of $\Phi_v(U)$ with respect to the metric $g_\eps$. By \Cref{minimplieslambdamin}, $\tilde u$ is a critical point of the functional 
     \begin{equation}\label{auxfunc}
          I(v)=A(v)+\int_U \tilde Hv,
     \end{equation}
 for a suitable $\tilde H\in C_{loc}^{1,\frac{\alpha}{4}}(U)$. Writing $A$ in local coordinates, and writing the Euler-Lagrange equation associated to \eqref{auxfunc} for $\tilde u$, one can argue exactly as in the proof of \Cref{uisc2} to conclude that $\tilde u\in C_{loc}^{2,\beta}(U)$ for any $\beta\in (0,1)$. Since $\tilde H\in C_{loc}^{1,\gamma}(U)$, the thesis follows as in \Cref{uisc2}.
\end{proof}
Let $\pi:\mathbb H^n\to \Om$ be the projection on the first $2n$ components and $u\in BV(\Om)$. We denote 
\begin{equation}\label{def:S}
\Om_{u,0}=\pi(\partial E_u\setminus\partial^\star E_u).
\end{equation}
We say that $S\subseteq\hh^n$ is locally a \emph{$Y_1$-graph} around  $p=(\bar x_1,\ldots,\bar x_n,\bar y_1,\ldots,\bar y_n,\bar t)\in S$ if there exists a neighborhood $U\subseteq\rr^{2n}$ of $\bar w=(\bar x_1,\ldots,\bar x_n,\bar y_2,\ldots,\bar y_n,\bar t+\bar x_1\bar y_1)\in\rr^{2n}$ and a function $\varphi\in C(U)$ such that
   \begin{equation*}
       S=\{(x_1,\ldots,x_n,\varphi(w),y_2,\ldots,y_n,t-x_1\varphi(w))\,:\,w=(x_1,\ldots,x_n,y_2,\ldots,y_n,t)\in U\}
   \end{equation*}
    (cf. \cite{MR1871966}). Identifying $\rr^{2n}$ with the hyperplane $\{y_1=0\}\subseteq \hh^n$, we follow \cite{MR2223801, MR2333095} and consider the vector fields on $U$ given by 
    \begin{equation}\label{def:W}
    W^{\varphi,\eps}_1=\frac{\partial}{\partial x_1}+2\varphi  T|_{U},\quad W^{\varphi,\eps}_j=X_j|_{U},\quad W^{\varphi,\eps}_{n+j-1}=Y_j|_{U}\quad\text{and}\quad W_{2n}^{\varphi,\eps}=\eps T|_{U},
\end{equation}
where $j=2,\ldots,n$. We denote by $\left(\z_j\right)^\star$ the adjoint operator of $\z_j$ with respect to $L^2(\rr^{2n})$, we write $W^{\varphi,\eps}=(W^{\varphi,\eps}_1,\ldots,W^{\varphi,\eps}_{2n})$ and we let $C^\varphi_\eps$ be the coefficient matrix associated to the family $\z$. Notice that for any $j=2,\ldots,2n$, it holds that
    \begin{equation}\label{eq:adj}
        \left(\z_1\right)^\star \phi=-\z_1 \phi-2\phi T|_U\varphi\qquad\text{and}\qquad \left(\z_j\right)^\star \phi=-\z_j \phi.
   \end{equation}
and hence 
\begin{equation*}
       \left(\z_j\right)^\star \z_{2n}=-\z_{2n}\z_j.
    \end{equation*}
The following standard first variation formula follows as in \cite[Section 3.1]{MR2333095}.
\begin{lemma}
Let $u\in BV(\Om)$ be such that $E_u$ is an $H$-minimizer for $P_\eps$ on $\Om\times\rr$, and set $S=\graf(u)$. Assume that $S$ is the $Y_1$-graph of  $\varphi\in C^2(U)$, where $U\subseteq\rr^{2n}$. Then 
    \begin{equation}\label{eq:intweak}
        -\sum_{j=1}^{2n}\int_U\frac{\z_j\varphi\left(\z_j\right)^\star\psi}{\sqrt{1+|W^{\varphi,\eps}\varphi|^2}}\,dw=\int_UH(x_1,,\ldots,x_n,\varphi(w),y_2,\ldots,y_n)\psi\,dw
    \end{equation}
for any $\psi\in C^\infty_c(U)$.
\end{lemma}


\begin{proposition}\label{prop:regu}
    Let 
    $H\in C^{1,\gamma}_{loc}(\Om)$ 
    and let $u\in BV(\Om)$ be such that $E_u$ is an $H$-minimizer for $P_\eps$ on $\Om\times\rr$. Then $u\in C^{3,\gamma}_{loc}(\Om\setminus \Om_{u,0})$ 
    and $\mathcal H^{s}(\Om_{u,0})=0$ for any $s>2n-7$. 
\end{proposition}
\begin{proof}
   Let $z=(\bar x,\bar y)\in\Om\setminus \Om_{u,0}$, where $\bar x=(\bar x_1,\ldots,\bar x_n)$ and $\bar y=(\bar y_1,\ldots,\bar y_n)$. By \Cref{reglambdamin}, $S=\graf (u)$ is a hypersurface of class $C^{3,\gamma}$ 
   near $p=(z,u(z))$, and $\mathcal{H}^{s}(\Om_{u,0})=0$ for any $s>2n-7$. Hence it suffices to show that $g_\eps(T, \ve)\neq 0$ on $\Om\setminus \Om_{u,0}$, where $\ve$ is the outer unit normal of $S$. Indeed, assume by contradiction that  $g_\eps(T, \ve(p))=0$. Then $\ve(p)\in \hhh_p$, and so $p$ is non-characteristic, since otherwise $\ve(p)\in\hhh_p=T_p S$. Therefore, by the implicit function theorem for intrinsic graphs (cf. \cite[Theorem 6.5]{MR1871966}) 
   $S$ is, locally near $p$, a $Y_1$-graph with respect to a continuous function $\varphi$ defined on an open neighborhood $U\subseteq\rr^{2n}$ of $\bar w=(\bar x_1,\ldots,\bar x_n,\bar y_2,\ldots,\bar y_n,\bar t+\bar x_1\bar y_1)$. Moreover, arguing as in \cite[Proposition 6.1]{ruled}, $\varphi\in C^{3,\gamma}(U)$. 
   Since $S$ is a $t$-graph vertical at $p$,  up to choosing a smaller $U$, we infer that
    \begin{equation}\label{premaxprin}
        T|_U\varphi(\bar w)=0\qquad\text{and}\qquad T|_U\varphi (w)\geq 0
    \end{equation}
    for any $w\in U$. 
If $W^{\varphi,\eps}$ is defined as in \eqref{def:W}, in the following we drop the superscript for the sake of clarity. Taking $\psi=W_{2n}\psi$ in \eqref{eq:intweak} and using \eqref{eq:adj} and the definition of adjoint operator, we infer that
\begin{equation*}
    \begin{split}
       \int_U\frac{\partial H}{\partial y_1}W_{2n}\varphi\psi\,dw&=-\int_UHW_{2n}\psi\,dw\\
       &=\sum_{j=1}^{2n}\int_U\frac{W_j\varphi W_j^\star \big(W_{2n}\psi\big)}{\sqrt{1+|W\varphi|^2}}\,dw\\
        =& \sum_{j=1}^{2n}\int_UW_{2n} \left(\frac{W_j\varphi}{\sqrt{1+|W\varphi|^2}}\right)W_{j}\psi\,dw\\
        =& \sum_{j,k=1}^{2n}\int_U\left\lbrace 
        \frac{W_{2n}\big(W_j\varphi\big)}{\sqrt{1+|W\varphi|^2}}
        -\frac{W_{2n}\big(W_{k}\varphi \big) W_{k}\varphi W_j\varphi}{\left(1+|W\varphi|^2\right)^{\frac{3}{2}}}  \right\rbrace W_{j}\psi\,dw.
    \end{split}
\end{equation*}
Setting $g=W_{2n}\varphi$, we get
\begin{equation*}
        0=\sum_{j,k=1}^{2n}\int_U\left\lbrace 
        \left(\frac{W_j^\star g}{\sqrt{1+|W\varphi|^2}}
        -\frac{W_k^\star g  W_{k}\varphi W_j\varphi}{\left(1+|W\varphi|^2\right)^{\frac{3}{2}}} \right)  W_{j}\psi+\frac{\partial H}{\partial y_1}g\psi\right\rbrace\,dw,
\end{equation*}
that can be rewritten as
\begin{equation*}
\int_UW\psi\cdot A\cdot\left(W g\right)^T+\langle W\psi,B\rangle g-\frac{\partial H}{\partial y_1}g\psi\,dw=0,
\end{equation*}
where $A= \left(A_{jk}\right)_{jk}$  and $B=(B_j)_j$ are defined by
\begin{equation*}
    A_{jk}=\frac{\left(1+|W\varphi|^2\right)\delta_{jk}-W_j\varphi W_k\varphi}{\left(1+|W \varphi|^2\right)^{\frac{3}{2}}}\qquad\text{and}\qquad  B_j= 2 T|_U\varphi A_{j1}
    \end{equation*}
for any $j,k=1,\ldots,2n$.
Since $W\varphi$ is continuous, an easy computation shows that, up to choosing a smaller $U$, there exists $\alpha>0$ such that $ \xi\cdot A(w)\cdot\xi^T\geq \alpha|\xi|^2$
for any $w\in U$ and any $\xi\in\rr^{2n}$. Moreover, as $W f=Df\cdot C^{\varphi,\eps}$
for any $f\in C^1(U)$, we infer that
\begin{equation*}
    \int_U D\psi\cdot \tilde A\cdot(Dg)^T+\langle D\psi,\tilde B\rangle g-\frac{\partial H}{\partial y_1}g\psi\,dw=0,
\end{equation*}
where
\begin{equation*}
    \tilde A=C^{\varphi,\eps}\cdot A\cdot \left(C^{\varphi,\eps}\right)^T\qquad\text{and}\qquad\tilde B=B \cdot \left(C^{\varphi,\eps}\right)^T.
\end{equation*}
Being $C^{\varphi,\eps}(w)$ invertible and continuous in $w$, up to restricting $U$ there exists $\tilde\alpha>0$ such that 
\begin{equation*}
\begin{split}
        \xi\cdot \tilde A\cdot\xi^T&=\xi\cdot C^{\varphi,\eps}\cdot A\cdot \left(C^{\varphi,\eps}\right)^T\cdot\xi^T=\xi\cdot C^{\varphi,\eps}\cdot A\cdot \left(\xi\cdot C^{\varphi,\eps}\right)^T\geq\alpha|\xi\cdot C^{\varphi,\eps}|^2\geq\tilde\alpha|\xi|^2
\end{split}
\end{equation*}
on $U$ for any $\xi\in\rr^{2n}$.
Hence $\tilde A$ is uniformly elliptic on $U$. Therefore, recalling \eqref{premaxprin} and choosing a suitable smaller neighborhood $U$, we can apply a strong maximum principle as in \cite[page 73, Theorem 10]{MR0762825} to conclude that $T|_U\varphi\equiv 0$
on $U$. In particular, $g_\eps(T,\ve)=0$ in a neighborhood $O$ of $p$ and  $\mathcal{H}^{2n-1}(V)>0$, where $V=\pi(O)$. Arguing \emph{verbatim} as in \cite[page 169]{MR0775682}, it follows that $V\subseteq \Om_{u,0}$, which is a contradiction with $\mathcal{H}^{2n-1}(\Om_{u,0})=0$.
\end{proof}
\begin{proposition}\label{wunoloc}
  Let $H\in C^{1,\gamma}_{loc}(\Om)$ and let $u\in BV(\Om)$ be such that $E_u$ is an $H$-minimizer for $P_\eps$ on $\Om\times\rr$. Then $u\in W^{1,1}(\Om).$
\end{proposition}

\begin{proof}
 By \Cref{prop:regu},  $u\in C^{3,\gamma}(\Om\setminus \Om_{u,0})$ and $\mathcal{H}^{2n-1}(\Om_{u,0})=0$, where  $\Om_{u,0}$ is defined in \eqref{def:S}. Let $\tilde Du$ be the distributional derivatives of $u$, and consider the decomposition $\tilde D u=Du\mathcal{L}^{2n}+(Du)_s$ as in \Cref{reprlemma}. It is enough to show that $(Du)_s\equiv 0$. 
 Since in particular $u\in W^{1,1}_{loc}(\Om\setminus\Om_{u,0})$, \eqref{periforw11} combined with \Cref{reprlemma} implies that
    \begin{equation*}\label{decoaprile}
    \begin{split}
        P_\varepsilon(E_u,\tilde \Om\times\rr)&=\int_{\tilde \Om}\sqrt{\varepsilon^2+|Du+X|^2}\,dz+(Du)_s(\tilde \Om)\\
        &=\int_{\tilde \Om\setminus \Om_{u,0}}\sqrt{\varepsilon^2+|Du+X|^2}\,dz+(Du)_s(\tilde \Om)\\
        &=P_\eps(E_u,(\tilde\Om\setminus\Om_{u,0})\times\rr)+(Du)_s(\tilde \Om),
    \end{split}
    \end{equation*}
    so that
    \begin{equation*}
        (Du)_s(\tilde \Om)=P_\eps(E_u,(\tilde \Om\cap \Om_{u,0})\times\rr).
    \end{equation*}
Arguing as in \Cref{linftyloc}, there exists $L>0$ such that $
        \partial^\star E_u\cap(\tilde \Om\times\rr)\subseteq \tilde \Om\times[-L,L].
$   Therefore, exploiting \eqref{ineq:perieq}, there exists $\tilde C>0$ such that
    \begin{multline*}
            (Du)_s(\tilde \Om)
            \leq   
            \tilde C P(E_u,(\tilde \Om\cap \Om_{u,0})\times\rr)= \tilde C \mathcal H^{2n}(\partial ^*E_u\cap(\tilde \Om\cap \Om_{u,0})\times\rr)\\
            \leq 2\tilde C L\mathcal H^{2n-1}(\Om_{u,0}) =0.\qedhere
    \end{multline*}
\end{proof}
\begin{proposition}\label{liploc}
    Let $\Om\subseteq\rr^{2n}$ be a bounded domain with Lipschitz continuous boundary and let $H\in C^{1,\gamma}_{loc}(\Om)$ for some $\gamma\in (0,1)$. 
    Let $u\in BV(\Om)$ be such that $E_u$ is an $H$-minimizer for $P_\eps$ on $\Om\times\rr$. then $u\in\lip_{loc}(\Om)$.
\end{proposition}
\begin{proof}
    Fix $z_0\in\Om$. Let $r_0>0$ small enough to ensure that $B_{r_0}=B(z_0,r_0)\Subset\Om$. We claim that there exists $r\in(r,r_0)$ such that \eqref{Giusti}, \eqref{subextremal} and \eqref{Serrinstrict} hold for $B_r=B(z_0,r)$. Indeed, setting $H_0=\|H\|_{L^\infty(B_{r_0})}$, then
    \begin{equation*}
        \left|\int_A H\,dz\right|\leq H_0|A|
    \end{equation*}
    for any $r\in (0,r_0)$ and any measurable set $A\subseteq B_r$. Therefore, recalling that
    \begin{equation*}
       h(B_r)=\frac{P(B_r)}{|B_r|}=\frac{1}{r}\frac{r_0P(B_{r_0})}{|B_{r_0}|}
    \end{equation*}
    (cf. \cite{MR3467379}), the claim concerning \eqref{Giusti} and \eqref{subextremal} follows by taking 
    $
    r< \frac{r_0P(B_{r_0})}{H_0|B_{r_0}|}
    $.
    Regarding \eqref{Serrinstrict}, it suffices to see that $H_{\partial B_r}=\frac{2n-1}{r}$, so that $H_{\partial B_r}$ can be made arbitrarily big as $r$ becomes small.
Given $w\in W^{1,1}(B_r)$, we consider the extension of $w$ to a function $\tilde w\in BV(\Om)$ by letting $\tilde w\equiv u$ on $\Om\setminus\overline{B_r}$. In particular, $u$ and $\tilde w$ share the same trace on $\partial \Om$. Using that $u\in W^{1,1}(\Om)$ is such that $E_u$ is an $H$-minimizer for $P_\eps$ on $\Om\times\rr$ and \eqref{trace}, we infer that $u|_{B_r}$ minimizes the functional
\begin{equation}\label{minisobdur}
I(w)= \int_{B_r}\sqrt{\eps^2+|Dw+X|^2}\,dz+\int_{B_r}Hwdz+\int_{\ptl B_r}|u-w|d\mathcal{H}^{2n-1}
\end{equation}
in $W^{1,1}(B_r)$. 
    Since $\mathcal{H}^{2n-1}(\Om_{u,0})=0$ by \Cref{prop:regu},
we can take a sequence of open sets  $( \Om_k)_k$ such that $ \Om_{u,0}\subseteq  \Om_{k+1}\subseteq  \Om_k\subseteq \Om $ for any $k\in\mathbb N$, $ \Om_{u,0}=\bigcap_{k\in\mathbb N} \Om_k$ and $\mathcal H^{2n-1}( \Om_k\cap\partial B_r)\to 0$ as $k\to\infty$.
For any $k\in\mathbb N$, let $\varphi_k\in C^{2,\gamma}(\overline{B_r})$ be such that $\varphi_k\equiv u$ on $\partial B_r\setminus  \Om_k$ and
\begin{equation}\label{unistimafi}
        \sup_{\partial B_r}|\varphi_k|\leq 2\sup_{\partial B_r}|u|.
    \end{equation}
We apply \Cref{existence*} to get a classical solution $v_k\in C^{2}(\overline{B_r})$ to \eqref{pmc'} such that $v_k\equiv\varphi_k$ on $\partial B_r$ for any $k\in\mathbb N$. In particular, being $\mathcal{I_\varepsilon}$ convex in $W^{1,1}(B_r)$, we infer that $v_k$ minimizes the functional
\begin{equation}\label{minisob}
I_k(w)= \int_{B_r}\sqrt{\eps^2+|Dw+X|^2}\,dz+\int_{B_r}Hwdz+\int_{\ptl B_r}|\varphi_k-w|d\mathcal{H}^{2n-1}
\end{equation}
in $ W^{1,1}(B_r)$.
Thanks to \Cref{lem:nonop} we can apply \cite[Proposition 4.4]{pmc1} which, together with \eqref{unistimafi}, implies that $(v_k)_k$ is uniformly bounded in $L^{\infty}(B_r)$. By \Cref{uisc2},  $v_k\in C^{3,\ga}_{loc}(B_r)$, and by \Cref{mainige'} the sequence $(v_k)_k$ is locally uniformly bounded in $\lip(B_r)$. By Ascoli-Arzelà Theorem, there exists $v\in \lip_{loc}(B_r)$ such that, up to a subsequence, $v_k\to v$ locally uniformly on $B_r$. Since $(v_k)_k$ is uniformly bounded in $L^\infty(B_r)$ and by the lower semicontinuity properties of \Cref{equifun}, we can pass to the limit in \eqref{minisob} with $w\equiv 0$ to infer that $v\in W^{1,1}(B_r)$. 
Let us check that $v$ is a minimum of $I$ in $W^{1,1}(B_r)$. Let $y\in\ptl B_r$ be a regular point for $u$, $k$ big enough so that $y\in \ptl B\setminus \Om_k$ and $V$ a neighborhood of $y$ in $\ptl B_r$. Let $\varphi^\pm\in C^{2,\gamma}(\overline{B_r})$ be such that
\begin{equation*}
\varphi^\pm= u \text{ in }V\qquad\text{and}\qquad
\varphi^-\leq \varphi_k\leq\varphi^+\text{ in }\ptl B_r,
\end{equation*}
and let $v^\pm$ be the solutions to the Dirichlet problem in $B_r$ with boundary datum $\varphi^\pm$. Then, using the maximum principle \cite[Theorem 10.7]{GT}, we get $v^-\leq v_k\leq v^+$,
and $v=u$ at regular points of $u$, which together with $\mathcal H^{2n-1}(\Om_{u,0})$ implies that $u$ and $v$ have the same trace on $\partial B_r$. 
Moreover, by the lower semicontinuity, 
the local uniform convergence of $v_h\to v$ and that $\varphi_k\to u$ in $L^1(\ptl B_r)$, we get
\begin{equation*}
I(v)\leq \liminf_{k\to+\infty}I_k(v_k)\leq \liminf_{k\to+\infty}I_k(w)=I(w)
\end{equation*}
for any $w\in W^{1,1}(B_r)$. Therefore, $u$ and $v$ are minimizers in $W^{1,1}(B_r)$ of the functional $I$. Recalling that $u-v\in W^{1,1}_0(B_r)$, the conclusion easily follows by the strict convexity of the functional.
\end{proof}

As a direct consequence of \Cref{existenceeps}, \Cref{uisc2} and \Cref{liploc}, we have the following result.

\begin{theorem}\label{fullregularity}
    Let $\Om\subseteq\rr^{2n}$ be a bounded domain with Lipschitz continuous boundary and let $H\in C^{1,\gamma}_{loc}(\Om)\cap L^\infty(\Om)$ for some $\gamma\in (0,1)$. Assume that \eqref{Giusti} and \eqref{subextremal} holds. Then there exists $u\in C_{loc}^{3,\gamma}(\Om)\cap W^{1,1}(\Om)$ which minimizes $\mathcal{I_\varepsilon}$ and solves \eqref{pmc'}. Moreover, if $H\in C^\infty(\Om)$ then $u\in C^\infty(\Om)$.
\end{theorem}

\subsection{Existence of minimizers: the extremal case}\label{extremalcases}
Throughout this subsection we fix $H\in \lip(\Om)\cap C^{1,\gamma}_{loc}(\Om)$ for some $\gamma\in(0,1)$.
To deal with solutions to \eqref{pmc'} in the extremal case \eqref{Giusti2}, we follow the approach of \cite{MR487722}. To this aim, we generalize the notion of $H$-minimizer as in \eqref{limminloc2esp} admitting merely measurable functions. 
\begin{definition}\label{genhsoldefeps}
    A measurable function $u:\Om\longrightarrow[-\infty,+\infty]$ is a \emph{generalized $H$-minimizer} for $P_\eps$ on $\Om\times\rr$ if \eqref{epsminmaggio} holds.
\end{definition}
According to the notation introduced in \cite{MR0775682}, if $u:\Om\longrightarrow[-\infty,+\infty]$ is a measurable function we set
\begin{equation}\label{npiunmeno}
    N_+=\{z\in\Om\,:\,u(z)=+\infty\}\qquad\text{and}\qquad N_-=\{z\in\Om\,:\,u(z)=-\infty\}.
\end{equation}
As in the Euclidean setting, we have the following minimization property.
\begin{proposition}\label{neminimoeps}
    Let $u$ be a generalized $H$-minimizer. Then 
    \begin{equation}\label{minimobassoeps}
        P(N_{\pm},A)\pm\int_{N_\pm\cap A}H\,dz\leq P(F,A)\pm\int_{F\cap A}H\,dz
    \end{equation}
    for any open set $A\Subset\Om$ and any measurable set $F\subseteq\Om$ such that $N_+\Delta F\Subset A$.
\end{proposition}
\begin{proof}
  Let us check this property for $N_+$, being the case $N_-$ analogous. For any $j\in\mathbb N$, set $u_j=u-j$. Then $u_j$ converges almost everywhere to 
    \begin{equation*}
v(z)=
\displaystyle{\begin{cases}
+\infty&\text{ if }z\in N_+\\
-\infty&\text{ if }z\notin N_+.
\end{cases}
}
\end{equation*}
Arguing \emph{verbatim} as in \cite[Section 21.5]{MR2976521}. (cf. \Cref{applisub} for the proof of a similar result), $E_v$ is an $H$-minimizer as in \eqref{epsminmaggio}. Assume by contradiction that there exists an open set $A\Subset\Om$, a measurable set $F$ such that $N_+\Delta F\Subset A$ and $\delta>0$ such that
\begin{equation*}
      P(F;A)+\int_{F\cap A}H\,dz\leq P(N_+;A)+\int_{N_+\cap A}H\,dz-\delta.
\end{equation*}
Given $L>0$, we set $A_L=A\times [-L,L]$ and $A_{2L}=A\times(-2L,2L)$, and 
    \begin{equation*}
F_L=
\displaystyle{\begin{cases}
F\times\rr&\text{ in }A_L\\
N_+\times\rr&\text{ otherwise.}
\end{cases}
}
\end{equation*}
Then, using \eqref{trace}, we have
\begin{equation*}
\begin{split}
    P_\eps(F_L;A_{2L})+\int_{ F_L\cap A_{2L}}H\,dx&\leq 4\int_\Om\sqrt{\eps^2+|X|^2}\,dz+ 2LP(N_+,A)+2LP(F;A)\\
    &\quad+2L\int_{N_+\cap A}H\,dz+2L\int_{F\cap A}H\,dz\\
    &\leq 4\int_\Om\sqrt{\eps^2+|X|^2}\,dz+4LP(N_+;A)+4L\int_{F_+\cap A}H\,dz-2L\delta\\
    &=P_\eps(E_v;A_{2L})+\int_{E_v\cap A_{2L}}H\,dx+4\int_\Om\sqrt{\eps^2+|X|^2}\,dz-2L\delta\\
    &<P_\eps(E_v;A_{2L})+\int_{E_v\cap A_{2L}}H\,dx,
    \end{split}
\end{equation*}
where the last strict inequality follows provided that $L>\frac{2}{\delta}\int_\Om\sqrt{\eps^2+|X|^2}\,dz$. Being $E_v$ an $H$-minimizer, a contradiction follows.
\end{proof}
\begin{proof}[Proof of \Cref{thm:1}]
    Let us assume \eqref{Giusti2}, since otherwise the thesis follows by \Cref{fullregularity}. Let $(\varepsilon_j)_j\subseteq(0,1)$ be such that $\varepsilon_j\searrow 0$ as $j\to\infty$. Let $(\Om_j)_j$ be a sequence of open domains with Lipschitz boundary such that $ \Om_j\Subset\Om_{k}\Subset\Om$  for any $j<k$m $\bigcup_{j=0}^\infty\Om_j=\Om$  and $P(\Om_j)\to P(\Om)$ as $j\to\infty$ (cf. \cite{MR3314116}).
    By hypothesis
      \begin{equation*}
          \Big|\int_{\Om_j}H\Big|<P(\Om_j)
      \end{equation*}
      for any $j\in\mathbb N$.
Hence, for any fixed boundary datum $\varphi_j\in L^1(\partial\Om_j)$, \Cref{existenceeps} implies the existence of $u_j\in BV(\Om_j)$ which minimizes $\mathcal I_\eps$ on $\Om_j$. 
      In view of \Cref{minimplieslambdamin}, $E_{u_j}$ is an $H_j$-minimizer for $P_{\varepsilon_j}$ on $\Om_j\times\rr$. Arguing as in \cite{MR487722,MR3767675}, up to vertical translations we assume that
      \begin{equation}\label{translatedhpeps}
          \min\{|\{z\in\Om_j\,:\,u_j(z)\leq 0\}|,\{z\in\Om_j\,:\,u_j(z)\geq 0\}|\}\geq \frac{|\Om|}{4}
      \end{equation}
      for any $j\in\mathbb N$. 
   Applying again a compactness argument as in \cite[Section 21.5]{MR2976521} and \Cref{applisub}, there exists a generalized $H$-minimizer $u$ for $P_{\eps}$ on $\Om\times\rr$ as in \Cref{genhsoldefeps} such that $u_j\to u$ almost everywhere on $\Om$. We are left to show that $u\in L^\infty_{loc}(\Om)$. 
   Indeed, in this case, since $E_u$ is a Caccioppoli set, \Cref{equifun} would imply that $u\in BV_{loc}(\Om)$, whence \Cref{liploc} and \Cref{uisc2} guarantee the requested regularity. To show that $u\in L^\infty_{loc}(\Om)$, we let $N_+$ and $N_-$ be as in \eqref{npiunmeno}. We claim that $|N_+|=|N_-|=0$. In this case, arguing \emph{verbatim} as in \cite[Proposition 16.7]{Giusti}, volume density estimates as in \Cref{prop:voldensity} allow to conclude that $u\in L^\infty_{loc}(\Om)$. We prove that $|N_+|=0$, being the other case analogous. To this aim, we apply \Cref{neminimoeps} to infer that $N_+$ is a minimizer as in \eqref{minimobassoeps}. But then, \eqref{Giusti2} allows to apply \cite[Lemma 1.2]{MR487722}, whence either $|N_+|=0$ or $|N_-|=|\Om|$. Being the latter possibility in contradiction with \eqref{translatedhpeps}, we conclude that $|N_+|=0$, from which the thesis follows.
\end{proof}

\section{Essential uniqueness of solutions}\label{sec:uniqueness}
Similarly to what happens in the Euclidean setting (cf. \cite{MR487722}), the extremal case \eqref{Giusti2}
describes those  maximal configurations $\Om$ for which \eqref{pmc'} admits a classical solution. Moreover, in these cases, solutions are unique up to vertical translations.
\begin{proof}[Proof of \Cref{thm:uniqueness}]
    The equivalence between (i) and (ii), thanks to \Cref{thm:1}, follows word-by-word the proof of \cite[Proposition 2.2]{MR487722}. To prove that (i) is equivalent to (iii) it suffices to notice that, if $u\in C^2(\Om)$ solves \eqref{pmc'} in $\Om$, then
    \begin{equation*}
        \int_{\Om_t}H\,dz=\int_{\Om_t}\divv\left(\frac{Du+X}{\sqrt{\eps^2+|Du+X|^2}}\right)\,dz=\int_{\partial \Om_t}\frac{\langle\nu_t,Du+X\rangle}{\sqrt{\eps^2+|Du+X|^2}}\,d\mathcal{H}^{2n-1},
    \end{equation*}
    so that
     \begin{equation*}
        \int_{\Om}H\,dz=\lim_{t\to 0^+}\int_{\partial \Om_t}\frac{\langle\nu_t,Du+X\rangle}{\sqrt{\eps^2+|Du+X|^2}}\,d\mathcal{H}^{2n-1}.
    \end{equation*}
    To show uniqueness, assume that $u,v\in C^2(\Om)$ satisfy \eqref{pmc'} in $\Om$. Up to changing the sign of $u$ and $v$, we can assume that $\int_\Om H\,dz\geq 0$. Being \eqref{pmc'} invariant under vertical translations, we fix $z_0\in\Om$ and we assume that $v(z_0)=u(z_0)$. 
    In order to simplify the notation, we let 
    \begin{equation*}
        W_\eps u=\frac{Du+X}{\sqrt{\eps^2+|Du+X|^2}},
    \end{equation*}
    and $W_\eps v$ accordingly. Notice in particular that $|W_\eps u|,|W_\eps v|\leq 1$.
    Notice that, for any $\varphi\in \lip(\Om)$ such that $\varphi\geq 0$ on $\Om$, and for any sufficiently small $t\in(0,1)$, it holds that
    \begin{equation*}
       \begin{split}
           \int_{\Om_t}\langle W_\eps u-W_\eps v,D\varphi\rangle\,dz=\int_{\partial\Om_t}\varphi\left(\langle \nu_t,W_\eps u\rangle-\langle\nu_t, W_\eps v\rangle\right)\,dz\geq \int_{\partial\Om_t}\varphi\left(\langle \nu_t,W_\eps u\rangle-1\right)\,dz.
       \end{split} 
    \end{equation*}
    For any $k>0$, let $\varphi_k=\max\{0,\min\{v-u,k\}\}$. Then $\varphi_k\in\lip(\Om)$, $\varphi_k\geq 0$ on $\Om$ and $\langle W_\eps u-W_\eps v,D\varphi_k\rangle\leq 0$ on $\Om$, so that
    \begin{equation*}
        0\geq  \int_{\Om_t}\langle W_\eps u-W_\eps v,D\varphi_k\rangle\,dz\geq-k\left(P(\Om_t)-\int_{\partial\Om_t}\langle \nu_t,W_\eps u\rangle\,dz\right).
    \end{equation*}
    Therefore, in view of (iii), we let first $t\to 0^+$ and then $k\to\infty$ to obtain
    \begin{equation}\label{intdisuguniq}
        \int_{\Om}\langle W_\eps u-W_\eps v,D\varphi\rangle\,dz=0,
    \end{equation}
    where $\varphi:=\max\{0,v-u\}$ verifies again $\langle W_\eps u-W_\eps v,D\varphi\rangle\leq 0$ on $\Om$. Hence \eqref{intdisuguniq} implies that $\langle W_\eps u-W_\eps v,D\varphi\rangle =0$ on $\Om$. In view of the definition of $\varphi$, a simple computation shows that $D\varphi=0$ on $\Om$, so that $\varphi(z)=\varphi(z_0)=0$ for any $z\in\Om$. Hence we conclude that $u\equiv v$ on $\Om$.
\end{proof}
\section{Some applications to the sub-Riemannian Heisenberg group}\label{applisub}
\subsection{Existence of local minimizers via Riemannian approximation}
Throughout this subsection, we fix a bounded domain $\Om\subseteq\rr^{2n}$ with Lipschitz boundary and $H\in L^\infty_{loc}(\Om)$. It is well-known (cf. e.g. \cite{MR2262784,pmc1}) that \eqref{pmc'} arises naturally as an elliptic approximation of the \emph{sub-Riemannian prescribed mean curvature equation} \eqref{pmchor}.
In this section we provide existence of solutions to \eqref{pmchor} in a broad sense by means of our previous results coupled with the aforementioned Riemannian approximation scheme (cf. e.g. \cite{MR0748865,MR2043961,MR2262784,MR2312336} for further insights). To describe our approach, assume that $u\in W^{1,1}(\Om)$ is a weak solution to \eqref{pmchor}. A standard variational argument, together with \cite[Theorem 3.2]{MR3276118} shows that
\begin{equation}\label{limminloc2}
\int_{\tilde\Om}|Du+X|+\int_{\tilde\Om}Hu\,dz\leq \int_{\tilde\Om}|Dv+X|+\int_{\tilde\Om}Hv\,dz
    \end{equation}
for any open set $\tilde\Om\Subset\Om$ and any $v\in BV_{loc}(\Om)$ such that $\{u\neq v\}\Subset\tilde\Om$, where, following \cite{MR3276118}, we have set 
\begin{equation}\label{oriperisubgr}
\int_{\tilde\Om}|Dv+X|:=P_\hh(E_v,\tilde\Om\times\rr)
\end{equation}
for any open set $\tilde\Om\Subset\Om$ and any $v\in BV_{loc}(\Om)$. A function $u\in BV_{loc}(\Om)$ satisfying \eqref{limminloc2} is called a \emph{$H$-minimizer for $P_\hh$}. This definition is motivated by the fact that, arguing as in the proof of \Cref{minimplieslambdamin} (cf. also \cite[Theorem 3.15]{MR3276118} and \cite[Corollary 3.16]{MR3276118}), the subgraph $E_u$ of an $H$-minimizer satisfies
\begin{equation}\label{horimin}
        P_\hh (E_u,A)+\int_{E_u\cap A}H\,dz\leq P_\hh(F,A)+\int_{F\cap A}H\,dz
    \end{equation}
    for any open set $A\Subset\Om\times\rr$ and any measurable set $F$ such that $E\Delta F\Subset A$. On the other hand, a truncation argument as in \cite[Theorem 14.8]{MR0775682} implies that if $E_u$ satisfies \eqref{horimin}, then $u$ is an $H$-minimizer for $P_\hh$. 
We stress that, in light of \cite[Theorem 1.2]{MR3276118}, the sub-Riemannian area functional in \eqref{oriperisubgr} is finite. 
Therefore, for any $\varphi\in L^1(\ptl\Om)$, we define the functional $\mathcal{I}_\hh:BV(\Om)\to\rr$ by
\[
\mathcal{I}_\hh(v)=P_\hh(E_v,\Om\times\rr)+\int_\Om Hu dz+\int_{\ptl\Om}|v-\varphi|d\mathcal{H}^{2n-1}.
\]
\begin{proposition}\label{nonextremalhoriz} Let $H\in L^\infty(\Om)$ and assume that \eqref{Giusti} and \eqref{subextremal} hold. Then $\mathcal{I}_\hh$
 has a minimum in $BV(\Om)\cap L^\infty_{loc}(\Om)$ for every $\varphi\in L^1(\partial\Om)$.  
\end{proposition}
\begin{proof}
    Let $B\subseteq\rr^{2n}$ be a ball containing $\Om$ such that the Euclidean distance between $\ptl\Om$ and $\ptl B$ is positive, and extend $H$ to $B$ by letting $H\equiv 0$ outside $\Om$. As done in the proof of \Cref{existenceeps}, minimizing $\mathcal{I}_\hh$ is equivalent to minimize 
    \[
\mathcal{J}_\hh(v)=P_\hh(E_v,\Om\times\rr)+Var(v,B\setminus\Om)+\int_BHvdz
    \]
    in $K=\{v\in BV(B):v=\phi\text{ in } B\setminus\Om \}$, where $\phi$ is a fixed function in $W_0^{1,1}(B)$ with trace $\varphi$ on $\ptl \Om$. Let $(\varepsilon_j)_j\subseteq(0,1)$ be such that $\varepsilon_j\searrow 0$ as $j\to\infty$. Let $v_j$ be a minimizer of the functional $\mathcal{J}_{\eps_j}$ defined in \eqref{def:functionalJ}. By \eqref{eq:BVbound} and \eqref{ineq:perieq}, we have
  \begin{multline*}
   \delta Var(v_j,B)\leq \mathcal{J}_{\eps}(v_j)+\tilde{C}\leq \mathcal{J}_{\eps_j}(\phi)+\tilde{C}
    \leq C'P(E_\phi,\Om\times\rr)+Var(\phi,B\setminus\Om)+\int_BH\phi dz +\tilde C,
    \end{multline*}
    where $\tilde C=\int_\Om |X|dz+\int_{\ptl\Om}|\varphi|d\mathcal{H}^{2n-1}$ and $C'=C'(\Om)$. Hence, arguing as in \Cref{existenceeps} $(v_j)_j$ is bounded in $BV(B)$ and it converges in $L^1(B)$ to $v_0\in K$. By \eqref{doublelsc}, $v_0$ is a minimizer of $\mathcal{J}_\hh$, whence $\mathcal{I}_\hh$ has a minimizer in $BV(\Om)$. Finally, the same arguments of \Cref{densityestimates} can be carried out thanks to the sub-Riemannian density estimates (c.f. \cite{psv}) to prove that $u\in L^\infty_{loc}(\Om)$. 
\end{proof}

As in the Riemannian setting, in the extremal case \eqref{Giusti2} we rely again on the notion of generalized solution. 
\begin{definition}\label{genhsoldef}
    A measurable function $u:\Om\longrightarrow[-\infty,+\infty]$ is a \emph{generalized $H$-minimizer} for $P_\hh$ on $\Om\times\rr$ if \eqref{horimin} holds.
\end{definition}
If $N_+$ and $N_-$ are defined as in \eqref{npiunmeno}, in view of \eqref{trace} the following analogous to 
\eqref{neminimoeps} holds.
\begin{proposition}\label{neminimo}
    Let $u$ be a generalized $H$-minimizer. Then 
    \begin{equation}\label{minimobasso}
        P(N_{\pm},A)\pm\int_{N_{\pm}\cap A}H\,dz\leq P(F,A)\pm\int_{F\cap A}H\,dz
    \end{equation}
    for any open set $A\Subset\Om$ and any measurable set $F\subseteq\Om$ such that $N_+\Delta F\Subset A$.
\end{proposition}
Before proving \Cref{thm:sub-Riem}, we need the following compactness argument, whose proof follows the approach of \cite{MR2976521}.
\begin{theorem}\label{hcompthm}
    Let $(\varepsilon_j)_j\subseteq(0,1)$ be such that $\varepsilon_j\searrow 0$ as $j\to\infty$. Let $(\Om_j)_j$ be a sequence of open sets such that $ \Om_j\Subset\Om_{k}\Subset\Om$  for any $j<k$ and $\bigcup_{j=0}^\infty\Om_j=\Om$. Let $(H_j)_j\subseteq L^\infty(\Om)$ be such that $H_j\to H$ uniformly on $\Om$. For any $j\in\mathbb N$, assume that $u_j\in BV_{loc}(\Om_j)$ is such that $E_{u_j}$ is an $H_j$-minimizer for $P_{\varepsilon_j}$ on $\Om_j\times\rr$ (cf. \Cref{hmindeff}). Then, up to a subsequence, $(u_j)_j$ converges almost everywhere to a generalized $H$-minimizer for $P_\hh$ on $\Om\times\rr$.
    If in addition $E_u$ is a Caccioppoli set, then
    \begin{equation}\label{pepsweakstarph}
        P_{\varepsilon_j}(E_{u_j},\cdot)\rightharpoonup^*P_{\mathbb H}(E_u,\cdot)
    \end{equation}
    locally on $\Om\times\rr$.
\end{theorem}
    
\begin{proof}
We set $E_j=E_{u_j}$ and $P_{\varepsilon_j}=P_j$ for any $j\in\mathbb N$. First we show that there exists an $\mathbb H$-Caccioppoli set $E$ on $\Om\times\rr$ such that, up to a subsequence,
  \begin{equation}\label{e0}
      \chi_{E_j}\to\chi_E
  \end{equation}
  in $L^1(A')$ for any open set $A'\Subset\Om\times\rr$. 
 Since $\bar A'$ is compact, there exists $k\in\mathbb N$, $p_1,\ldots,p_k\in A'$, $r_1,\ldots,r_k>0$ and $\tilde\Om\Subset\Om$ such that
  \begin{equation*}
      \bar A'\Subset \bigcup_{i=1}^kB(p_i,r_i)\Subset\tilde\Om\times\rr.
  \end{equation*}
  Therefore, in view of \Cref{pepestoph} and \eqref{ineq:perden}, we have
  \begin{equation*}
  \begin{split}
      P_{\mathbb H}\left(E_j,\bigcup_{i=1}^kB(p_i,r_i)\right)\leq P_{\varepsilon_j}\left(E_j,\bigcup_{i=1}^kB(p_i,r_i)\right)\leq\sum_{i=1}^k P_{\varepsilon_j}(E_j,B(p_i,r_i))
      \leq c_2\sum_{i=1}^kr_i^{2n}.
  \end{split}
  \end{equation*}
  Hence we can apply \Cref{firstcompact} to infer the existence of a finite $\mathbb H$-perimeter set $F$ in $\bar A'$ such that, up to a subsequence, $\chi_{E_j}\to\chi_{F}$ in $L^1( A')$. Taking a sequence of relative compact sets that covers $\Om\times \rr$ and using a standard diagonal argument, \eqref{e0} follows. Thanks to \cite[Lemma 16.3]{MR0775682}, $E=E_u$ for a measurable function $u:\Om\longrightarrow[-\infty,+\infty]$, and moreover, up to a subsequence, $u_j\to u$ almost everywhere on $\Om$. 
  Arguing as above,
\begin{equation*}
    \sup_{j\in\mathbb N}P_{j}(E_j,K)<\infty
\end{equation*}
for any $K\Subset\Om\times\rr$. Therefore, \cite[Theorem 1.59]{MR1857292} implies the existence of a $(2n+1)$-valued Radon measure $\mu$ on $\Om\times\rr$ and a scalar Radon measure $\lambda$ on $\Om\times\rr$ such that, up to a subsequence,
\begin{equation}\label{twoweakstar}
D_{j}\chi_{E_j}\rightharpoonup^*\mu\qquad\text{and}\qquad P_{j}(E_j,\cdot)\rightharpoonup^*\lambda
\end{equation}
locally on $\Om\times\rr$ as $j\to\infty$.
We claim that 
\begin{equation}\label{misureuguali}
    \mu=(D_{\mathbb H}\chi_E,0).
\end{equation}
 Indeed, fix $K\Subset\Om\times\rr$ and $g\in C^1_c(K,\rr^{2n+1})$, and set
 \begin{equation*}
V_{j}=\sum_{i=1}^n(g_iX_i+g_{n+i}Y_i)+g_{2n+1}\varepsilon_j T=:V+g_{2n+1}\varepsilon_j T.
 \end{equation*}
 Then, since $\chi_{E_j}\to\chi_E$ in $L^1_{loc}(\Om\times\rr)$, we infer that
\begin{equation*}
\begin{split}
    \left|\int_Kg\cdot dD_{j}\chi_{E_j}-\int_Kg\cdot d(D_{\mathbb H}\chi_E,0)\right|&=\left|\int_{K\cap E_j}\divv V_j\,dx-\int_{K\cap E}\divv V\,dx\right|\\
    &  
    \leq\int_{K\cap(E\Delta E_j)}|\divv V|\,dx+\varepsilon_j\int_K\left|\frac{\partial g_{2n+1}}{\partial t}\right|\,dx\\
    &\leq |K\cap(E\Delta E_j)|\|\divv V\|_{L^\infty(K)}+\varepsilon_j|K|\left\|\frac{\partial g_{2n+1}}{\partial t}\right\|_{L^\infty(K)}\\
    &\to 0
\end{split}
\end{equation*}
as $j\to\infty$. An easy approximation argument allows to extend the previous convergence to any $g\in C^0_c(K,\rr^{2n+1})$, thus proving \eqref{misureuguali}. Therefore, combining \eqref{twoweakstar}, \eqref{misureuguali} and \cite[Proposition 1.62]{MR1857292}, we conclude that
\begin{equation}\label{disugmeas}
    P_{\mathbb H}(E,\cdot)\leq\lambda(\cdot).
\end{equation}
Let $A$ be an open set such that $A\Subset\Om\times\rr$, and let $F$ be a Caccioppoli set on $\Om\times\rr$ such that $F\Delta E\Subset A$. We claim that\begin{equation}\label{stronghminthm}
   P_{\mathbb H} (E,A)+\int_{E\cap A}H\,dx\leq P_{\mathbb H} (F,A)+\int_{F\cap A}H\,dx.
    \end{equation}
For any $j\in\mathbb N$, Let $A_j$ be an open set with Lipschitz boundary such that 
$ E\Delta F\Subset A_0\Subset A_j\Subset A_{j+1}\Subset A $
and $\bigcup_{j\in\mathbb N}A_j=A$.
Up to a further subsequence, we assume that $\chi_{E_j}\to\chi_E$ almost everywhere on $\Om\times\rr$. This fact allows to choose $(A_j)_j$ in such a way that 
\begin{equation}\label{charonbound}
    \chi_{E_j}\to\chi_E
\end{equation}
$\mathcal{H}^{2n}$-almost everywhere on $\partial A_j$. Moreover, being $E_j$ and $F$ are Caccioppoli sets, 
\begin{equation}\label{radonisgood}
    \mathcal H^{2n}(\partial^*E_j\cap \partial A_j)=0\qquad\text{and}\qquad \mathcal H^{2n}(\partial^*F\cap \partial A_j)=0
\end{equation}
for any $j\in\mathbb N$.
For any $j\in\mathbb N$, we define $ F_j:=(F\cap A_j)\cup(E_j\setminus A_j).$
It is clear that $F_j$ is a Caccioppoli set such that $F_j\Delta E_j\Subset A_j\Subset A$. 
Therefore in particular 
\begin{equation}\label{e1}
    F_j\cap (A_j\setminus A_0)=E_j\cap (A_j\setminus A_0).
\end{equation}
Notice that, thanks to \eqref{ineq:perieq},\eqref{charonbound} and \eqref{radonisgood} and arguing as in \cite[Theorem 21.14]{MR2976521}, \begin{equation}\label{e2}
\lim_{j\to\infty}P_j(F_j,\partial A_j)=0.
\end{equation}
We are able to prove \eqref{stronghminthm}. Indeed, exploiting the  $H$-minimality of $E_j$ for $P_j$, together with \eqref{e1}, we see that 
\begin{equation*}
\begin{split}
 P_j(E_j,A)+\int_{E_j\cap A}H_j\,dx&\leq P_j(F_j,A)+\int_{F_j\cap A}H_j\,dx\\
&=P_j(F_j,A_j)+P_j(F_j,\partial A_j)+P_j(F_j,A\setminus\overline{A_j})+\int_{F_j\cap A}H_j\,dx\\
&=P_j(F,A_j)+P_j(F_j,\partial A_j)+P_j(E_j,A\setminus\overline{A_j})+\int_{F_j\cap A}H_j\,dx\\
&\leq P_j(F,A)+P_j(F_j,\partial A_j)+P_j(E_j,A\setminus\overline{A_j})+\int_{F_j\cap A}H_j\,dx,
 \end{split}
\end{equation*}
which implies that
\begin{equation*}
    P_j(E_j,A_j)+\int_{E_j\cap A}H_j\,dx\leq P_j(F,A)+P_j(F_j,\partial A_j)+\int_{F_j\cap A}H_j\,dx.
\end{equation*}
Therefore, exploiting \Cref{pepestoph} together with \eqref{e0} and \eqref{e1}, we can pass to the limit and obtain 
\begin{equation*}
    \lambda(A)+\int_{E\cap A}H\,dx\leq P_{\mathbb H}(F,A)+\int_{F\cap A}H\,dx.
\end{equation*}
Notice that \eqref{disugmeas}
allows to achieve \eqref{stronghminthm}. Finally, if $E_u$ is a Caccioppoli set, then we can choose $F=E$ in the previous inequality, so that, recalling \eqref{disugmeas}, \eqref{pepsweakstarph} follows.
\end{proof}
\begin{proof}[Proof of \Cref{thm:sub-Riem}]
     Being the non-extremal case already covered by \eqref{nonextremalhoriz}, we can assume \eqref{Giusti2}. Let $(\varepsilon_j)_j$ and $(\Om_j)_j$ be as in the statement of \Cref{hcompthm}. Assume in addition that $P(\Om_j)\to P(\Om)$ as $j\to\infty$. Assume that, for any $j\in\mathbb N$, $\Om_j$ has Lipschitz boundary. 
   By hypothesis
     \begin{equation*}
          \Big|\int_{\tilde\Om_j}H\Big|<P(\Om_j)
      \end{equation*}
      for any $j\in\mathbb N$.
      Arguing as in the proof of \Cref{existence*},
    there exists a sequence $(H_j)_j\subseteq C^\infty(\overline\Om)$ such that
\begin{equation*}\label{approxhlipdue}
    H_j\to H \text{ uniformly on }\Om\qquad\text{and}\qquad\|H_j\|_{C^1(\Om)}\leq H_0+\|H\|_{L^\infty(\Om)}+1
\end{equation*}
for any $j\in\mathbb N$, being $H_0$ the Lipschitz constant of $H$ on $\Om$. Arguing again as in the proof of \Cref{existence*}, up to a subsequence each $H_j$ satisfies the hypotheses of \Cref{existenceeps} on $\Om_j$.
    Hence, for any $j\in\mathbb N$, there exists $u_j\in BV(\Om_j)$ which minimizes $\mathcal I_{\varepsilon_j}$, with source $H_j$, for any $j\in\mathbb N$. In view of \Cref{minimplieslambdamin}, $E_{u_j}$ is an $H_j$-minimizer for $P_{\varepsilon_j}$ on $\Om_j\times\rr$.
   Arguing as in the proof of \Cref{thm:1}, \Cref{hcompthm} implies the existence of a generalized $H$-minimizer for $P_\hh$ on $\Om$ as in \Cref{genhsoldef} such that, up to vertical translations, $u_j\to u$ almost everywhere on $\Om$. Moreover, arguing again as in \Cref{thm:1}, in view of \Cref{neminimo} and exploiting sub-Riemannian volume density estimates as in \cite{psv}, we conclude that $u\in L^\infty_{loc}(\Om)$. Hence, being $E_u$ an $H$-Caccioppoli set, then \cite[Theorem 1.2]{MR3276118} implies that $u\in BV_{loc}(\Om)$. In this case, $E_u$ is a Caccioppoli set, so that \eqref{pepsweakstarph} follows. 
\end{proof}

\subsection{An application to the sub-Riemannian second variation formula}
Let $S=\partial E$ be a smooth non-characteristic hypersurface, let $A\subseteq\hh^n$ be an open set such that $A\cap S\neq\emptyset$, and let $\xi\in C^\infty_c(A)$. Then it is well-known (cf. e.g. \cite{MR2322147,MR3319952}) that
    \begin{equation*}
        \frac{d^2}{dt^2}|\partial E_t|_\hh(A)\Big\vert_{t=0}=\int_S\left(|\nabla^{\mathbb H,S}\xi|^2-\xi^2\left(q-(H^{\mathbb H})^2\right)\right)\,d|\partial E|_\hh,
    \end{equation*}
    where
    \begin{equation}\label{q}
        q=\sum_{h,k=1}^{2n}Z_h(\vh_k)Z_k(\vh_h)+4\langle J(\vh),\nabla^\hh(Td^\hh)\rangle+4n(Td^\hh)^2
    \end{equation}
and where by $E_t$ we mean a smooth variation along the vector field $\xi\vh$.
Observe that $q$ does not depend on the chosen unitary extension of $\vh|_S$. Here $d|\partial E|_\hh$ is the horizontal perimeter measure, $d^\hh$ is the horizontal Carnot-Carathéodory distance defined in \eqref{dccaprile}, $\nabla^{\mathbb H}$ is the \emph{horizontal gradient}, $\nabla^{\mathbb H,S}$ is the \emph{horizontal tangential gradient} of $S$ and $H^{\mathbb H}$ is the \emph{horizontal mean curvature} of $S$ (cf. \cite{MR2354992}).
In the following theorem, we recover the interpretation of $q$ given in \cite{MR2710215,MR2723818} as the limit of  $|h^\eps|^2+\ric_\eps(\v)$ as $\eps \to 0$, using the fact that the family of Riemannian structures $(\hn,g_\eps)$ approximate in a suitable sense the sub-Riemannian manifold $(\hn,\langle\cdot,\cdot\rangle$) (cf. \cite{MR2312336}). Notice that both terms $|h^\eps|^2$ and $\ric_\eps(\v)$ diverge, while the sum is controlled thanks to \eqref{needed inproposition}.
\begin{theorem}\label{thm:2v}
   Let $S=\partial E$ be a hypersurface of class $C^3$. Then
   \begin{equation*}
       \lim_{\eps\to 0}\left(|h^\eps|^2+\ric_\eps(\v)\right)=q
   \end{equation*}
   locally uniformly in the non-characteristic part of $S$, where $q$ is as in \eqref{q}.
\end{theorem}
\begin{proof}
    It suffices to observe that
    \begin{equation*}
        \v=\frac{\vh}{\sqrt{1+\eps^2(Td^\hh)^2}}+\frac{\eps Td^\hh}{\sqrt{1+\eps^2(Td^\hh)^2}}\eps T.
    \end{equation*}
    Therefore, noticing that
    \begin{equation*}
        \lim_{\eps\to 0}\v_j=\vh_j\qquad\text{and}\qquad\lim_{\eps\to 0}\frac{\vt}{\eps}=Td^\hh
    \end{equation*}
    locally uniformly for any $j=1,\ldots,2n$, and moreover
    \begin{equation*}
        \lim_{\eps\to 0}Z_i\v_j=Z_i\vh_j\qquad\text{and}\qquad\lim_{\eps\to 0}Z_{2n+1}\v_k=\lim_{\eps\to 0}Z_k\vt=0
    \end{equation*}
    locally uniformly for any $i,j=1,\ldots,2n$ and $k=1,\ldots,2n+1$, the thesis follows from \eqref{needed inproposition}.
\end{proof}

\section{Proofs of \Cref{sec:preliminaries} and \Cref{sec:geomprop}} \label{sec:proofs}


\begin{proof}[Proof of \eqref{ricciexpress}]
   Fix $j=1,\ldots,n$ and $i=1,\ldots,2n+1$. From \eqref{levi-civita}, it easily follows that
\begin{equation*}
\begin{split}
& g_\eps \big(\nabla _{Z_i} U, X_j\big)=Z_iu_j-\frac{1}{\eps}(\delta_{i,j+n}u_{2n+1}+\delta_{i,2n+1}u_{j+n})\\
& g_\eps \big(\nabla _{U} U, X_j\big)=U u_j-\frac{2}{\eps}u_{n+j}u_{2n+1}\\
& g_\eps \big(\nabla_{X_j}U, \nabla _{U} X_j\big)=\frac{1}{\eps}X_j\big( u_{n+j}u_{2n+1}\big)+\frac{1}							{\varepsilon^2}\big(u^2_{2n+1}-u_{n+j}^2\big).
\end{split}
\end{equation*}
Hence
    \begin{equation*}
    \begin{split}
  g_\varepsilon\left(\n{X_j}\n{U}U,X_j\right)&=X_j g_\eps\big(\nabla_U U,X_j\big)=X_j(U u_j)-                             												\frac{2}{\varepsilon}X_j\big( u_{n+j}u_{2n+1}\big)\\
-g_\varepsilon\left(\n{U}\n{X_j}U,X_j\right)&=-U\g{\n{X_j}U,X_j}+\g{\n{X_j}U,\n{U}X_j}\\
            &=-U(X_ju_j)+\frac{1}{\eps}X_j\big( u_{n+j}u_{2n+1}\big)+\frac{1}							{\varepsilon^2}\big(u^2_{2n+1}-u_{n+j}^2\big).
 \end{split}
 \end{equation*}
Moreover, from
    \begin{equation*}
      [U,X_j]=\sum_{k=1}^{2n+1}u_k[Z_k,X_j]-X_j(u_k)Z_k=2u_{n+j}T-\sum_{k=1}^{2n+1}X_j(u_k)Z_k,
    \end{equation*}
    we get
    \begin{equation*}
    \begin{split}
     g_\varepsilon\left(\n{[U,X_j]}U,X_j\right)&=\frac{2}{\eps}u_{n+j}\g{\n{\eps T}U,X_j}-\sum_{k=1}^{2n+1}X_j(u_k)\g{\n{Z_k}U,X_j}\\
       &=2u_{n+j}T(u_j)-\frac{2}{\eps^2}u_{n+j}^2-\sum_{k=1}^{2n+1}X_j(u_k)Z_k(u_j)+\frac{1}{\eps}X_j\big( u_{n+j}u_{2n+1}\big)\\
       &=[U,X_j](u_j)-\frac{2}{\eps^2}u_{n+j}^2+\frac{1}{\eps}X_j\big( u_{n+j}u_{2n+1}\big),
       \end{split}
    \end{equation*}
    so that 
    \begin{equation*}
    g_\varepsilon\left(\n{X_j}\n{U}U-\n{U}\n{X_j}U+\n{[U,X_j]}U,X_j\right)=\frac{1}{\eps^2}(u_{2n+1}^2-3u_{n+j}^2).
    \end{equation*}
    A similar computation shows that
    \begin{equation*}
    \begin{split}
    &g_\varepsilon\left(\n{Y_j}\n{U}U-\n{U}\n{Y_j}U+\n{[U,Y_j]}U,Y_j\right)=\frac{1}{\eps^2}(u_{2n+1}^2-3u_{j}^2)\\
         & g_\varepsilon\left(\n{\eps T}\n{U}U-\n{U}\n{\eps T}U+\n{[U,\eps T]}U,\eps T\right)=\frac{1}{\eps^2}\sum_{j=1}^n( u_j^2+u_{n+j}^2).
    \end{split}
    \end{equation*}
    In view of the previous computations, \eqref{ricciexpress} follows.
\end{proof}

\begin{proof}[Proof of \eqref{riemsomma}]
Consider the extension of $\v$ to a neighborhood of $S$ as $\v=\nabla d$. By a continuity argument and since $S\setminus S_0$ is dense in $S$, it suffices to show \eqref{riemsomma} in the non-characteristic part of $S$. Let us fix $p\in S\setminus S_0$. It is well known (cf. e.g. \cite{MR3385193}) that there exist a local orthonormal frame $e_1,\ldots,e_{n-1},e_{n+1},\ldots,e_{2n-1}$ of $\mathcal H\cap T_pS$ such that $e_{n+j}=J(e_j)$ at $p$ for any $j=1,\ldots,n-1$. We consider the extension to a local orthonormal frame of $T_pS$ by letting 
    \begin{equation*}
        e_n=\frac{1}{\sqrt{1-(\v_{2n+1})^2}}J(\v)\qquad\text{and}\qquad e_{2n}=-\frac{\vt}{\sqrt{1-(\vt)^2}}(\v)_h+\sqrt{1-(\vt)^2}\eps T,
    \end{equation*}
    where $(\v)_h=\v-\vt\eps T$ is the horizontal projection.
Moreover, for any $i,j=1,\ldots,2n$ we set  $e_i=\sum_{k=1}^{2n+1}\alpha^i_kZ_k$. From our choice it follows that
\begin{equation}\label{weird2}
1= -\escpr{e_i,J (e_{i+n})}=\sum_{k=1}^n(\alpha^{n+i}_{n+k}\alpha^i_k-\alpha^{n+i}_k\alpha^i_{n+k})
\end{equation}
for any $i=1,\ldots, n-1$. We set
$$a_{i,j}=-\sum_{l,s=1}^{2n+1}\alpha^i_s\alpha^j_lZ_s(\v_l).$$
 Using \eqref{levi-civita}, it follows that
\begin{equation*}
\begin{split}
    &\g{\n{e_i}e_{n+i},\v}=a_{i,n+i}-\frac{\vt}{\eps}\qquad \qquad\g{\n{e_{n}}e_{2n},\v}=a_{n,2n}-\frac{1}{\eps}\\
    &\g{\n{e_{n+i}}e_{i},\v}=a_{n+i,i}+\frac{\vt}{\eps}\qquad\qquad \g{\n{e_{2n}}e_n,\v}=a_{2n,n}-\frac{1}{\eps}+\frac{2(\vt)^2}{\eps}
\end{split}
\end{equation*}
for any $i=1,\ldots,2n-1$, and 
\[
\g{\n{e_i}e_j,\v}=a_{i,j}
\]
for any $i,j=1,\ldots,2n,\ |i-j|\neq n$. Writing
\begin{equation*}
\begin{split}
    b_{i,n+i}&=-\frac{\vt}{\eps}\qquad\qquad\qquad b_{n,2n}=-\frac{1}{\eps},\qquad\\
    b_{n+i,i}&=\frac{\vt}{\eps} \qquad\qquad\qquad b_{2n,n}=-\frac{1}{\eps}+\frac{2(\vt)^2}{\eps}
    \end{split}
\end{equation*}
for any $i=1,\ldots,n-1$, and $b_{i,j}=0$ for any $i,j=1,\ldots, 2n$ with $|i-j|\neq n$, it follows that
\begin{equation}\label{normsquare}
     |h_p|^2=\sum_{i,j=1}^{2n}h_{i,j}h_{j,i}=\sum_{i,j=1}^{2n}a_{i,j}a_{j,i}+2\sum_{i=1}^n(a_{i,n+i}b_{n+i,i}+a_{n+i,i}b_{i,n+i})+2\sum_{i=1}^nb_{i,n+i}b_{n+i,i},
\end{equation}
where $h_{i,j}$ are the entries of $h$ with respect to $e_1,\ldots,e_{2n}$. First, by \eqref{propv1} and arguing as in \cite{ruled},
\begin{equation}\label{uhh1}
    \begin{split}
        \sum_{i,j=1}^{2n}a_{i,j}a_{j,i}=\sum_{l,s=1}^{2n+1}Z_s(\v_l)Z_l(\v_s).
    \end{split}
\end{equation}
Moreover,
\begin{equation}\label{uhh2}
        \sum_{i=1}^nb_{i,n+i}b_{n+i,i}=-(n-1)\frac{(\vt)^2}{\eps^2}+\frac{1}{\eps}\left(\frac{1}{\eps}-\frac{2(\vt)^2}{\eps}\right)=-(n+1)\frac{(\vt)^2}{\eps^2}+\frac{1}{\eps^2}.
\end{equation}
Finally,
\begin{equation}\label{uhh3}
    \sum_{i=1}^n(a_{i,n+i}b_{n+i,i}+a_{n+i,i}b_{i,n+i})=\frac{\vt}{\eps}\sum_{i=1}^{n-1}\left(a_{i,n+i}-a_{n+i,i}\right)-\frac{1}{\eps}\big(a_{2n,n}+(1-2(\vt)^2)a_{n,2n}\big).
\end{equation}
%
From one hand, in view of \eqref{propv1}, \eqref{propv2}, \eqref{propv3} and 
\eqref{weird2},
\begin{equation}\label{uh1}
    \begin{split}
       \frac{\vt}{\eps} \sum_{i=1}^{n-1}\left(a_{i,n+i}-a_{n+i,i}\right)&=\frac{\vt}{\eps}\sum_{i=1}^{n-1}\left(\sum_{h,k=1}^{2n}\alpha^{n+i}_k\alpha^i_hZ_k(\v_h)-\sum_{h,k=1}^{2n}\alpha^{n+i}_k\alpha^i_hZ_h(\v_k)\right)\\
        &=\frac{2(\vt)^2}{\eps^2}\sum_{i=1}^{n-1}\left(-\sum_{k=1}^n\alpha^{n+i}_k\alpha^i_{n+k}+\sum_{k=1}^n\alpha^{n+i}_{n+k}\alpha^i_k\right)\\
        &=2(n-1)\frac{(\vt)^2}{\eps^2}.
    \end{split}
\end{equation}
Again by \eqref{propv1}, \eqref{propv2} and \eqref{propv3}, a similar computation implies that
\begin{equation}\label{uh2}
    \begin{split}
       -\frac{1}{\eps}&\big(a_{2n,n}+(1-2(\vt)^2)a_{n,2n}\big)
=2\left\langle J(\v),\nabla\left(\frac{\vt}{\eps}\right)\right\rangle+\frac{2(\vt)^2}{\eps^2}.
    \end{split}
\end{equation}
Inserting \eqref{uh1} and \eqref{uh2} in \eqref{uhh3} we get
\begin{equation}\label{uhh3'}
\sum_{i=1}^n(a_{i,n+i}b_{n+i,i}+a_{n+i,i}b_{i,n+i})=2n\frac{(\vt)^2}{\eps^2}+2\left\langle J(\v),\nabla\left(\frac{\vt}{\eps}\right)\right\rangle.
\end{equation}
Replacing \eqref{uhh1}, \eqref{uhh2} and \eqref{uhh3'} in \eqref{normsquare} we obtain the result.
\end{proof}


\begin{proof}[Proof of \eqref{lapbelt}]
   Fix $p\in S$ and consider a geodesic frame $e_1,\ldots,e_{2n}$ of $T_p S$ at $p$ (cf. \cite{MR1138207}), i.e. a local orthonormal frame of $TS$ such that $ \nabla^S_{e_i}e_j(p)=0$
    for any $i,j=1,\ldots,2n+1$, where $\nabla^S$ is the Levi-Civita connection of $(S,g_\eps|_S)$. Let us set $e_i=\sum_{k=1}^{2n+1}\alpha_k^iZ_k$ for any $i=1,\ldots,2n+1$. From the definition of the basis it follow the relations
    \begin{equation}\label{compus}
        \sum_{k=1}^{2n+1}\alpha_k^i\alpha_k^j=\delta_{i,j},\qquad\sum_{k=1}^{2n+1}\alpha_k^i\ve_k=0\qquad\text{and}\qquad\sum_{k=1}^{2n}\alpha_l^k\alpha_s^k=g^{l,s}
    \end{equation}
    for any $i,j=1,\ldots,2n$ and any $l,s=1,\ldots,2n+1$. With respect to a geodesic frame (cf. \cite{MR1138207}) we have that  $\Delta_S  f(p)=\sum_{i=1}^{2n}e_ie_if(p)$. Hence
    \begin{equation}\label{lap1}
        \begin{split}
            \Delta_S  f       &
            =\sum_{i=1}^{2n}g_\eps( \nabla _{e_i}e_i,\nabla   f)+\sum_{i=1}^{2n}g_\eps( e_i,\nabla _{e_i}\nabla   f)\\
            &=\sum_{i=1}^{2n}(g_\eps( \nabla^S _{e_i}e_i,\nabla   f)+g_\eps(\nabla   f,\ve)g_\eps( \nabla_{e_i}e_i,\ve))+\sum_{i=1}^{2n}g_\eps( e_i,\nabla _{e_i}\nabla   f)\\
            &=\sum_{i=1}^{2n}-Hg_\eps(\nabla  f,\nu^\hh)+\sum_{i=1}^{2n}g_\eps( e_i,\nabla _{e_i}\nabla   f).
        \end{split}
    \end{equation}
    On the other hand, by \eqref{compus} and \eqref{invmetric}, we have
        \begin{equation}\label{lap2}
    \begin{split}
\sum_{i=1}^{2n}g_\eps( e_i,\nabla _{e_i}\nabla   f)=\sum_{l,s=1}^{2n+1}g^{l,s}g_\eps(Z_l,\nabla _{Z_s}\nabla   f)=\sum_{l,s=1}^{2n+1}g^{l,s}\Big(Z_s(Z_lf)-g_\eps(\nabla _{Z_s}Z_l,\nabla   f)\Big).
           \end{split}
    \end{equation}
    Exploiting \eqref{levi-civita}, we see that
    \begin{multline}\label{lap3}
            -\sum_{l,s=1}^{2n+1}g^{l,s}g_\eps(\nabla _{Z_s}Z_l,\nabla   f)=\sum_{l,s=1}^{2n+1}\ve_s\ve_lg_\eps(\nabla _{Z_s}Z_l,\nabla   f)\\
                        \begin{split}
            =&\sum_{k=1}^{n}\Big(-\ve_k\ve_{n+k}g_\eps(\nabla  f,T)+\frac{\vet}{\eps}\ve_kg_\eps(\nabla  f,Y_k)+\ve_k\ve_{n+k}g_\eps(\nabla  f,T)\\
            &-\frac{\vet}{\eps}\ve_{n+k}g_\eps(\nabla  f,X_k)+\frac{\vet}{\eps}\ve_kg_\eps(\nabla  f,Y_k)-\frac{\vet}{\eps}\ve_{n+k}g_\eps(\nabla  f,X_k)\Big)\\
            =&\frac{2\vet}{\eps}\sum_{k=1}^n\left(-\ve_{n+k}X_kf+\ve_kY_kf\right).
        \end{split}
    \end{multline}
    Therefore, \eqref{lapbelt} follows from \eqref{lap1}, \eqref{lap2} and \eqref{lap3}.
\end{proof}

\begin{proof}[Proof of \eqref{precodige}]
    Since $\Delta_S \vet$ depends only on $\vet|_S$, we extend $\ve|_S$ letting $\ve=\nabla  d $, in particular $\eps Td=\vet$.
Moreover, in view of \eqref{meancurvexpr}, we extend $H$ to a neighborhood of $S$ by letting 
\begin{equation}\label{hasdiv}
    H(p)=\sum_{i=1}^{2n+1}Z_i\ve_i(p).
\end{equation}
    Using \eqref{lapbelt}, $ TZ_i=Z_i T$, \eqref{propv1}, \eqref{invmetric} and \eqref{hasdiv}, 
and recalling that $$g_\eps(\nabla_SH,Z_{2n+1})=Z_{2n+1}(H)-\ve_{2n+1}g_\eps(\nabla H,\ve),$$
it holds that
\begin{equation}\label{eq:compu1}
        \begin{split}
        \Delta_S \vet-\frac{2\nu_{2n+1}^\eps}{\eps}&\langle\nabla (\eps Td), J\nu^\eps\rangle=\sum_{i,j=1}^{2n+1}g^{i,j}Z_i(Z_j (\eps Td))-Hg_\eps(\nabla  (\eps Td),\ve)\\
            &=\sum_{i,j=1}^{2n+1}g^{i,j}Z_{2n+1}(Z_i\ve_j)-H\sum_{j=1}^{2n+1}(Z_{2n+1}\ve_j) \ve_j \\
            &=Z_{2n+1}\left(\sum_{i=1}^{2n+1}Z_i\ve_i\right)-\sum_{i,j=1}^{2n+1}Z_{2n+1}(Z_i\ve_j)\ve_i\ve_j\\
            &=g_\eps(\nabla^SH,Z_{2n+1})+\ve_{2n+1}g_\eps(\nabla H,\ve)-\sum_{i,j=1}^{2n+1}Z_{2n+1}(Z_i\ve_j)\ve_i\ve_j.
        \end{split}
    \end{equation}
    By \eqref{propv1} and \eqref{weird1}, we have
    \begin{multline}\label{eq:compu0}
              \sum_{i,j=1}^{2n+1}Z_{2n+1}(Z_i\ve_j)\ve_i\ve_j=\sum_{i,j=1}^{2n+1}\left(\ve_iZ_{2n+1}\Big(Z_i\ve_j\ve_j\Big)-(Z_{2n+1}\ve_j)\Big(Z_i\ve_j\ve_i\Big)\right) \\
        =\frac{2\ve_{2n+1}}{\eps}\sum_{j=1}^{2n+1}Z_{2n+1}(Z_j d)J(\ve)_j=\frac{2\ve_{2n+1}}{\eps}\escpr{\nabla  (\eps Td),J(\ve)}.
    \end{multline}
   Inserting \eqref{eq:compu0} in \eqref{eq:compu1},  we get
    \begin{equation*}
        \Delta_S \vet=g_\eps(\nabla^SH,Z_{2n+1})+\ve_{2n+1}g_\eps(\nabla H,\ve).
    \end{equation*}
    Moreover, by \eqref{propv2} and \eqref{propv3},
    \begin{equation*}
        \begin{split}
            g_\eps(\nabla H,\ve)&=\sum_{i,j=1}^{2n+1}Z_j(Z_i\ve_i)\ve_j\\
            &=\sum_{i,j=1}^{2n+1}Z_i(Z_j\ve_i)\ve_j+2\sum_{j=1}^n \big(T(X_jd)(Y_jd)-T(Y_jd)(X_jd)\big)\\
            &=\sum_{i=1}^{2n+1}Z_i\left(\sum_{j=1}^{2n+1}Z_j\ve_i\ve_j\right)-\sum_{i=1}^{2n+1}Z_i\ve_jZ_j\ve_i-2\left\langle\nabla  \left(\frac{\ve_{2n+1}}{\eps}\right),J(\ve)\right\rangle,
        \end{split}    
    \end{equation*}   
    and, from \eqref{weird1}, we get
    \begin{equation*}    
        \begin{split}
            \sum_{i=1}^{2n+1}Z_i\left(\sum_{j=1}^{2n+1}Z_j\ve_i\ve_j\right)&=-2\sum_{i=1}^{2n+1}Z_i\left(\frac{\ve_{2n+1}}{\eps}J(\ve)_i\right)\\
            &=-2\frac{\ve_{2n+1}}{\eps}\sum_{i=1}^{2n+1}Z_i(J(\ve)_i)-2\left\langle\nabla  \left(\frac{\ve_{2n+1}}{\eps}\right),J(\ve)\right\rangle\\
            &=-4n\left(\frac{\ve_{2n+1}}{\eps}\right)^2-2\left\langle\nabla  \left(\frac{\ve_{2n+1}}{\eps}\right),J(\ve)\right\rangle.
        \end{split}
    \end{equation*}  
    The thesis then follows from \eqref{needed inproposition}.
\end{proof}

\bibliographystyle{abbrv}
\bibliography{biblio}
\end{document}